\documentclass[11pt]{article}
\usepackage[a4paper,total={6.5in,9in}]{geometry}
\usepackage{mathtools,booktabs}
\usepackage{advdate}
\usepackage{amsmath,amsfonts,amssymb}
\usepackage{amsthm}
\usepackage{bbm}
\usepackage[hidelinks]{hyperref}

\usepackage{setspace}
\usepackage{verbatim} 
\theoremstyle{plain}
\newtheorem{theorem}{Theorem}[section]
\newtheorem{proposition}[theorem]{Proposition}
\newtheorem{lemma}[theorem]{Lemma}

\theoremstyle{definition}
\newtheorem{definition}[theorem]{Definition}

\newtheorem{remark}[theorem]{Remark}

\allowdisplaybreaks[4]

\begin{document}

\title{\vspace*{-2.3cm}Large deviation principle for stochastic
	reaction-diffusion equations with super-linear drift on $\mathbb{R}$
	driven by space-time white noise}

\author{Yue Li$^{1}$, Shijie Shang$^{1}$,
Jianliang Zhai$^{1}$}
\footnotetext[1]{\, School of Mathematics, University of Science and Technology of China, Hefei, China. Email:liyue27@mail.ustc.edu.cn (Yue Li), sjshang@ustc.edu.cn (Shijie Shang), zhaijl@ustc.edu.cn (Jianliang Zhai).}

\maketitle

\begin{abstract}
In this paper, we consider stochastic reaction-diffusion equations with super-linear drift on the real line $\mathbb{R}$
driven by space-time white noise. A Freidlin-Wentzell large deviation principle is established by a modified weak convergence method on the space $C([0,T], C_{tem}(\mathbb{R}))$. 
Obtaining the main result in this paper is challenging  due to the setting of unbounded domain, the space-time white noise, and the superlinear drift term without dissipation.
To overcome these difficulties, the special designed norm on $C([0,T], C_{tem}(\mathbb{R}))$,
one order moment estimates of the stochastic convolution and two nonlinear Gronwall-type inequalities play an important role.

\end{abstract}
	

\medskip


\section{Introduction}
\setcounter{equation}{0}

The aim of this paper is to establish Freidlin-Wentzell type large deviation principles (LDPs) for the solutions of the following  stochastic reaction-diffusion equation (SRDE) on the real line $\mathbb{R}$ perturbed by space-time white noise
\begin{align}\label{eqn1}
	\begin{cases}
		\mathrm{d}u^\epsilon(t,x)= \frac{1}{2}\Delta u^\epsilon(t,x)\mathrm{d}t+b(u^\epsilon(t,x))\mathrm{d}t+\sqrt{\epsilon}\sigma(u^\epsilon(t,x))W(\mathrm{d}t,\mathrm{d}x),\ t\in[0,T], \ x\in\mathbb{R},\\
		u^\epsilon(0,x)=u_0(x),
	\end{cases}
\end{align}
as the small parameter $\epsilon>0$ converges to 0, where $\Delta=\partial_{xx}$,
$W$ is a space-time white noise defined on some filtered probability space $(\Omega,\mathcal{F},\{\mathcal{F}_t\}_{t\ge 0},\mathbb{P})$ satisfying the usual conditions. Here, $u_0,b$ and $\sigma$ are given functions from  $\mathbb{R}$ to $\mathbb{R}$. For specific conditions on these functions, please refer to Section \ref{sec 2}.

Throughout this paper, we will refer to the LDP of Freidlin-Wentzell type as simply "LDP" for convenience. The study of LDPs is partly motivated by the following. The LDP describes the limiting behavior of the laws of solutions when the noise in the equation converges to zero, in terms of a rate function. It allows researchers to identify and analyze rare but extreme events. The LDP problems arise in the theory of statistical
inference quite naturally. Large deviation techniques
are precious tools to study the exit time from a neighborhood of an asymptotically stable equilibrium point, the exit place determination or the transition between two equilibrium
points in randomly perturbed dynamical systems, which is important in the fields of statistical
and quantum mechanics, chemical reactions etc, see e.g. \cite{23, 31}.

There are many papers on LDPs for SRDEs driven by space-time white noise, and all of those papers concerned the case of bounded domains. The LDP for SRDEs was first obtained by  Freidlin \cite{F}, in which the space variable takes value on a unit circle, with $b$ assumed to be globally Lipschitzian and of linear growth, and $\sigma$ being a constant. Later in \cite{Sowers}  Sowers improved the result in \cite{F} to the case that $\sigma$ is globally Lipschitzian.
Cerrai and R\"ockner \cite{CR} obtained large deviation estimates for SRDEs with globally Lipschitz but unbounded diffusion coefficients, while assuming the reaction terms to be only locally Lipschitz with polynomial growth, and the LDP could be  uniform over bounded sets of initial data. Salins \cite{Salins} recently obtained uniform large deviations results for  SRDEs in a more general setting than Cerrai and R\"ockner’s work, where the LDP is uniform over unbounded sets of initial data. The assumptions on the reaction terms are relaxed to be written as the sum of a decreasing function and a Lipschitz continuous function, thereby removing the assumptions in \cite{CR} about the local Lipschitz continuity and polynomial growth rate of the reaction terms. Additional references on LDPs for SRDEs driven by space-time white noise on bounded domains can be found in the aforementioned papers.  However, to the best of our knowledge, there are no results on LDPs for SRDEs on unbounded domain driven by space-time white noise, even for the case that the coefficients $b$ and $\sigma$ are globally Lipschitz.  The dearth of existing research on the LDP for SRDEs on unbounded domains driven by space-time white noise provides a strong impetus for our investigation and this paper presents the first
result on this topic.

The study of SRDEs on unbounded domain driven by space time white noise poses greater challenges than the case of bounded domain and/or coloured noise.
As is well known, the solutions to SRDEs driven by space-time white noise like (\ref{eqn1}) usually are not semi-martingales, and therefore in particular It$\mathrm{\hat{o}}$'s formula could not be used.
In fact, even for the  simplest case: $b=0$, $\sigma=1$ and $u_0=0$, the supremum norm of the solution will explode in the following sense,
\begin{align}\label{230421.1116}
	\mathbb{P}\left(\sup_{x\in\mathbb{R}}|u(t,x)|=\infty\right)=1, \quad \forall\ t>0.
\end{align}
This phenomenon renders the usual stopping time argument, commonly used in the study of bounded domains, ineffective for equation (\ref{eqn1}).

The purpose of this paper is to investigate LDPs for (\ref{eqn1}). Our main result is based on the assumptions that the drift term $b$ can be locally Log-Lipschitz (see (\textbf{H1(b)})) and superlinear, and the well-posedness of (\ref{eqn1}) under the same assumptions is established in \cite{SZ}.
A typical example satisfying these conditions is $b(u)=u\log|u|$ for $u\in\mathbb{R}$.
As we point out above, since the space variable of the equation is on the real line $\mathbb{R}$ rather than a finite interval,
the supremum of solutions over $\mathbb{R}$ explodes (see  (\ref{230421.1116})).
Therefore, the usual stopping time argument (see e.g. \cite{DKZ}) could not be applied here.
To overcome this difficulty,
we appeal to a special designed norm initially raised in \cite{SZ}, of form
\begin{align}
	\sup_{t\leq T, x\in\mathbb{R}} \left(|u(t,x)|e^{-\lambda |x|e^{\beta t}}\right)
\end{align}
where, unlike the usual norm on $C_{tem}(\mathbb{R})$, the exponent also depends on time $t$ in a particular way. Then we can establish some necessary estimates under this norm.
Moreover, the logarithmic nonlinearity forces us to deal with moments of order one, which are harder to estimate than high-order moments.
It is noteworthy that our work does not impose assumptions on the dissipativity of $b$, making it essentially different from existing results (see e.g. \cite{CR,Salins}), where $b$ is locally Lipschitz continuous with polynomial growth and satisfies certain dissipative conditions, such as $b(z)=z-z^3$ which has the effect of “pulling the solution back toward the origin”. Obtaining the main result in this paper is challenging due to the setting of unbounded domain, the presence of  space-time white noise, and the superlinear drift term without dissipation.

To obtain the LDP, we adopt the weak convergence approach
and apply the criterion given in \cite{MSZ}, an adaption of the classical criteria of Budhiraja-Dupuis \cite{BD2008}.
To this end, we first establish the well-posedness of the corresponding skeleton equation.
Secondly, we show the strong continuity of deterministic skeleton equations with respect to controls under the uniform norm. The novelty is that we truncate the space $\mathbb{R}$ into a bounded domain such that outside this bounded domain the controlled term is sufficiently small, and in that bounded domain we exploit the uniform equicontinuity of the controlled term with respect to time and space variables to deduce a finite sum, such that we can use the weak convergence of controls.
Finally,  we prove the convergence to zero in probability of differences between the stochastic controlled equations and the random skeleton equations, as the noise intensity tends to zero.
Due to the setting of unbounded domain, the  space-time white noise, and the superlinear drift term without dissipation, technical difficulties mentioned above
have to be overcomed in the proof. To overcome these
difficulties, the special designed norm on $C([0,T], C_{tem}(\mathbb{R}))$,
one order moment estimates of the stochastic convolution and two nonlinear Gronwall-type inequalities play an important role. The whole program is non-trival.

Finally, we point out that there are a few results on the well-posedness of SRDEs driven by space-time white noise on unbounded domains; see \cite{Mueller, DaZ96, RZ99, AM03, MP11, MMP14, MN, FN21, MMR21}. In all of these literature, the coefficient $b$ satisfies certain dissipative conditions or has linear growth to ensure a global solution.
We believe that the ideas proposed in this paper
could be adapted to obtain the LDPs for SRDEs driven by space-time white noise on unbounded domains concerned in the papers mentioned above, which appear also to be open problems.


The rest of the paper is organized as follows. In Section 2, we present the framework and recall the large deviation theory, and then we state the main results of this paper and sketch their proofs. Section 3 is devoted to demonstrating the existence and uniqueness of the skeleton equations. In Section 4, we verify the two sufficient conditions in the weak convergence approach under locally Log-Lipschitz and logarithmic growth conditions on the drift $b$.  Finally, in Section 5, we complement the proof of the LDP for SRDEs under globally Lipschitz conditions.

\vskip 0.5cm

Convention on constants. Throughout the paper $C$ denotes a positive constant whose value may change from line to line. All other constants will be denoted by $C_1$, $C_2$, ... They are all positive and their precise values are not important. The dependence of constants on parameters if needed will be indicated, e.g, $C_T$, $C_{\lambda}$.

\section{Preliminaries and Main Results}\label{sec 2}
\setcounter{equation}{0}

In this section we first introduce some notations and collect a number of
preliminary results. Afterwards, we will present the main result of this paper along with an outline of its proof.

\subsection{Preliminaries}
For any $\lambda>0$ and $f\in C(\mathbb{R})$, denote
\begin{align*}
	|f|_{(-\lambda)}:=\sup_{x\in\mathbb{R}} \left(|f(x)|e^{-\lambda |x|}\right).
\end{align*}
Let
\begin{align*}
	C_{tem}:=\{f\in C(\mathbb{R}): |f|_{(-\lambda)}<\infty \text{ for any } \lambda >0 \}.
\end{align*}
The set $C_{tem}$ endowed with the metric
\begin{align*}
	d(f,g):=\sum_{n=1}^{\infty} \frac{1}{2^n} \min\{ 1,|f-g|_{(-\frac{1}{n})} \}
\end{align*}
is a Polish space. It is easy to see that for any $f,f_{n}\in C_{tem},n\ge 1$, $\lim_{n\rightarrow\infty}d(f_{n},f)= 0$ iff $\lim_{n\rightarrow\infty}|f_n-f|_{(-\lambda)}= 0$ for all $\lambda>0$.
Denote $C([0,T], C_{tem})$ by the space of functions $F:[0,T]\rightarrow C_{tem}$, continuous with respect to the metric of $C_{tem}$. The space $C([0,T], C_{tem})$ endowed  with the metric
\begin{align}
	\tilde{d}(F, G) = \sup_{t\leq T} d\big(F(t), G(t)\big)
\end{align}
is also a Polish space. Moreover,
\begin{align}
	\lim_{n\rightarrow \infty}\tilde{d}(F_n, F)= 0 \Longleftrightarrow \lim_{n\rightarrow \infty}\sup_{t\leq T , x\in \mathbb{R}} \big( |F_n(t,x) - F(t,x)| e^{-\lambda|x|} \big) =0 , \quad \forall\ \lambda>0.
\end{align}
Let $C([0,T]\times\mathbb{R},\mathbb{R})$ be the space of all real-valued continuous functions on $[0,T]\times\mathbb{R}$. Endowed with the topology of convergence on every compact subsets of $[0,T]\times\mathbb{R}$, $C([0,T]\times\mathbb{R},\mathbb{R})$ is a Polish space.

\vskip 0.5cm

In this paper we impose the following assumptions on the functions $b:\mathbb{R}\rightarrow \mathbb{R}$ and $\sigma:\mathbb{R}\rightarrow \mathbb{R}$.

\textbf{Hypothesis 0}

(\textbf{H0(a)})
The functions $b$ and $\sigma$ are  linear growth, that is, there exists a negative constant $L$ such that
	\begin{align}\label{221126.2141-01}
		|b(u)| + |\sigma(u)| \leq L(1+|u|),  \quad \forall\ u\in\mathbb{R}.
	\end{align}

(\textbf{H0(b)})
The functions $b$ and $\sigma$ are Lipschitz continuous, that is, there exists a negative constant $L$ such that
	\begin{align}\label{221126.2141}
		|b(u)-b(v)|+|\sigma(u)-\sigma(v)|\leq L|u-v|, \quad \forall\ u,v\in\mathbb{R}.
	\end{align}

Set $\log_+ u=\log(u\vee 1)$ for any $u\ge 0$.

\textbf{Hypothesis 1}

(\textbf{H1(a)}) The function $b$ is continuous, and there exist nonnegative constants $c_1$ and $c_2$ such that, for any $u\in\mathbb{R}$,
\begin{align*}
	|b(u)|\le c_1|u| \log_+ |u|+c_2.
\end{align*}

(\textbf{H1(b)}) There exist nonnegative constants $c_3$, $c_4$ and  $c_5$ such that, for any $u, v\in\mathbb{R}$,
\begin{align*}
	|b(u)-b(v)|\le c_3|u-v|\log_+\frac{1}{|u-v|}+c_4\log_+(|u|\vee|v|)\,|u-v|+c_5|u-v|.
\end{align*}

(\textbf{H1(c)}) The function $\sigma$ is globally Lipschitz and bounded, that is, there exist nonnegative constants $L_\sigma$ and $ K_\sigma$ such that, for any $u, v\in\mathbb{R}$
\begin{align*}
	&|\sigma(u)-\sigma(v)|\le L_\sigma |u-v|,\\
	&|\sigma(u)|\le K_\sigma.
\end{align*}

\begin{remark}
A straightforward calculation shows that (\textbf{H0(b)}) and (\textbf{H1(b)}) imply (\textbf{H0(a)}) and (\textbf{H1(a)}), respectively.
\end{remark}
\begin{remark}
A typical example of
function $b$ that satisfies (\textbf{H1(b)}) is the
	the function $x\in\mathbb{R}\mapsto x\log|x|\in\mathbb{R}$, more precisely,
	\begin{align*}
		\lvert x\log|x|-y\log|y|\rvert \le |x-y|\log\frac{1}{|x-y|}+\left[\log_+(|x|\vee|y|)+1+\log2\right] |x-y|,\quad \forall\  x,y\in\mathbb{R}.
	\end{align*}
See \cite[Example 2.3]{SZ} for more details.
\end{remark}

Now, let us introduce the definition of the solution to (\ref{eqn1}).

\begin{definition}\label{def solution}
	A random field solution to the equation (\ref{eqn1}) is a jointly measurable and $(\mathcal{F}_t)$-adapted space-time process $u^\epsilon:=\{u^\epsilon(t,x): \,(t,x)\in[0,T]\times\mathbb{R}\}$ such that, for every $(t,x)\in[0,T]\times\mathbb{R}$,
	\begin{align}
		u^\epsilon(t,x)=&P_tu_0(x)+\int_{0}^{t}\int_{\mathbb{R}} p_{t-s}(x,y)b(u^\epsilon(s,y))\mathrm{d}s\mathrm{d}y\nonumber\\
		&+\int_{0}^{t} \int_{\mathbb{R}}p_{t-s}(x,y)\sigma(u^\epsilon(s,y))W(\mathrm{d}s,\mathrm{d}y),\quad \mathbb{P}\text{-a.s.,}
	\end{align}
where $p_{t}(x,y):=\frac{1}{\sqrt{2\pi t}}e^{-\frac{(x-y)^2}{2t}}$, and $\{P_t\}_{t\ge 0}$ is the corresponding heat semigroup on $\mathbb{R}$. 
\end{definition}
\begin{remark}
	The above mild form is equivalent to  the weak solution to (\ref{eqn1}) in the sense of PDE; see \cite{Shiga,Iwata}.
\end{remark}

Given this definition, we have the following result; see \cite[Theorem 2.5]{SZ} and \cite[Theorem 2.2]{Shiga}.

\begin{proposition} \label{wellposedeqn1}
	Assume that $u_0\in C_{tem}$ and one of the following two conditions holds:

\begin{center}
(1) {\rm(\textbf{H0(b)})}  holds. \ \ \ \ \   (2) {\rm(\textbf{H1(b)})} and {\rm(\textbf{H1(c)})} hold.
\end{center}
Then
there exists a unique solution $u^\epsilon$ to the equation {\rm(\ref{eqn1})} in the sense of  Definition \ref{def solution}. Moreover,
$$
u^\epsilon\in C([0,T],C_{tem}),\ \ \mathbb{P}\text{-a.s.}
$$
\end{proposition}

\vskip 0.5cm

\subsection{Main Results}

%

Large deviations is concerned with exponential decay of probabilities of rare events. The aim of this paper is to establish LDP for the solutions $\{u^\epsilon,\epsilon> 0\}$ to {\rm(\ref{eqn1})} as $\epsilon$ tends to 0. To formulate our results in this paper, we first recall the definition of LDP. 

\begin{definition}
	(Rate function) A function $I(\cdot):C([0,T],C_{tem})\to [0,\infty]$ is called a rate function if for every constant $M<\infty$, $\{z\in C([0,T],C_{tem}):I(z)\le M\}$ is a compact subset of $C([0,T],C_{tem})$.
\end{definition}
It is well-known that $I$ is lower semi-continuous on the space $C([0,T],C_{tem})$. For a
subset $A\subset C([0,T],C_{tem})$, set $I(A):= \inf_{z\in A} I(z)$.

\begin{definition}
	(Large deviation principle) The solutions $\{u^\epsilon,\epsilon> 0\}$ to {\rm(\ref{eqn1})} is said to satisfy a LDP on $C([0,T],C_{tem})$ with the rate function $I$ if the following two conditions hold:
	
	(a) (large deviation upper bound) for each closed subset $F$ of $C([0,T],C_{tem})$,
	\begin{align*}
		\limsup_{\epsilon\to0} \epsilon\log \mathbb{P}(u^\epsilon\in F)\le -I(F).
	\end{align*}

	(b) (large deviation lower bound) for each open subset $G$ of $C([0,T],C_{tem})$,
	\begin{align*}
		\liminf_{\epsilon\to0} \epsilon\log \mathbb{P}(u^\epsilon\in G)\ge -I(G).
	\end{align*}
\end{definition}

For simplicity, let $\mathcal{H}:=L^2([0,T]\times\mathbb{R})$  equipped with  the usual $L^2$ norm, denoted by $|\cdot|_{\mathcal{H}}$, i.e.
$$
|h|_{\mathcal{H}}=\left( \int_{0}^{T}\int_{\mathbb{R}}|h(t,x)|^2\mathrm{d}t\mathrm{d}x \right) ^{\frac{1}{2}}.
$$
And set
\begin{align*}
	\mathcal{H}_N:=\left\{ h\in \mathcal{H}: |h|_{\mathcal{H}}=\left( \int_{0}^{T}\int_{\mathbb{R}}|h(t,x)|^2\mathrm{d}t\mathrm{d}x \right) ^{\frac{1}{2}}\leq N \right\}
\end{align*}
for any $N>0$. It is known that $\mathcal{H}_N$ endowed with weak topology is a Polish space. We then define a collection of random fields
\begin{align*}
	\mathcal{X}:=&\left\{h:[0,T]\times\mathbb{R}\times \Omega\rightarrow\mathbb{R}:h\text{ is a } \mathcal{F}_t \text{-predictable random field with }\right.\\
	&\ \left.h(\omega)\in \mathcal{H} \text{ and }|h(\omega)|_{\mathcal{H}}<\infty,\ \  \mathbb{P}\text{-a.e.}  \right\}.
\end{align*}
Also, for each $N>0$ define
\begin{align}\label{X_N}
	\mathcal{X}_N:=\left\{ h\in\mathcal{X} :|h(\omega)|_{\mathcal{H}}\le N,\ \  \mathbb{P}\text{-a.e.} \right\}.
\end{align}

We introduce the following so-called skeleton equation: for any $h\in \mathcal{H}$,
\begin{align}\label{ske0}
	\begin{cases}
	\mathrm{d}Y^h(t,x)= \frac{1}{2}\Delta Y^h(t,x)\mathrm{d}t+b(Y^h(t,x))\mathrm{d}t+\sigma(Y^h(t,x))h(t,x)\mathrm{d}t,\ t\in[0,T], \ x\in\mathbb{R},\\
	Y^h(0,x)=u_0(x)\in C_{tem}.
	\end{cases}
\end{align}
In Section \ref{221107.1422}, we will prove that there exists a unique solution $Y^h\in C([0,T],C_{tem})$ to (\ref{ske0}).

Now, we state the main result in this paper.

\begin{theorem}\label{221107.1426}
Assume the hypotheses of Proposition \ref{wellposedeqn1}. The solutions
 $\{u^\epsilon,\epsilon>0\}$ to {\rm(\ref{eqn1})} satisfy a LDP on $C([0,T],C_{tem})$ as $\epsilon$ tends to 0 with the rate function $I$ given by
 	\begin{align}\label{ratef}
		I(f):=\inf_{\left\{h\in\mathcal{H}:f=Y^h \right\}} \left\{\frac{1}{2} \int_{0}^{T}\int_{\mathbb{R} } |h(t,x)|^2\mathrm{d}x\mathrm{d}t\right\}, \quad \forall\ f\in C([0,T], C_{tem}),
	\end{align}
with the convention $\inf{\emptyset}=\infty$, here $Y^h$ solves \eqref{ske0}.
\end{theorem}

\begin{proof}[Proofs of Theorem \ref{221107.1426}]

The proof is divided into  two parts. 	

The first part is to prove the well-posedness of the skeleton equation (\ref{ske0}), denoted by $Y^h$; see Proposition \ref{Propo well posed SE 02} in Section \ref{221107.1422}. Therefore, there exists a mapping
$\mathcal{G}^0:C([0,T]\times\mathbb{R},\mathbb{R})\to C([0,T],C_{tem})$ such that for each $h\in\mathcal{H}$, $\mathcal{G}^0(\text{Int}(h))=Y^h$.
Here, for any $h\in\mathcal{H}$,
\begin{align}\label{221105.2016}
	\text{Int}(h)(t,x):= \int_{0}^{t}\int_{0}^{x}h(s,y)\mathrm{d}s\mathrm{d}y, \quad \forall\  (t,x)\in [0,T]\times \mathbb{R},
\end{align}
which induces a mapping denoted by Int from $\mathcal{H}$ to $C([0,T]\times \mathbb{R},\mathbb{R})$.
%


Now we state the second part.

We denote the product space of countably infinite copies of the real line by $\mathbb{R}^\infty$. Endowed with the topology of coordinate-wise convergence, $\mathbb{R}^\infty$ is a Polish space. Thus a sequence of independent standard Brownian motions $\beta=\{\beta_i\}_{i\in\mathbb{N}}$ can be regarded as a random variable with values in $C([0,T], \mathbb{R}^\infty)$.  Similarly, the Brownian sheet corresponding to the space-time white noise on $[0,T]\times\mathbb{R}$ can be regarded as a random variable with values in $C([0,T]\times\mathbb{R},\mathbb{R})$.
Let $\{e_i\}_{i=1}^{\infty}$ be an orthonormal basis in $L^2(\mathbb{R})$.
Given any Brownian sheet $\{W(t,x): (t,x)\in [0,T]\times\mathbb{R}\}$,
$\beta=\{\beta_i\}_{i\in\mathbb{N}}$ defined by
\begin{align}\label{230413.1626}
	\beta_i(t)= \int_0^t\int_{\mathbb{R}}e_i(y)W(\mathrm{d}s,\mathrm{d}y).
\end{align}
is a sequence of independent standard Brownian motions, where $W(\mathrm{d}s, \mathrm{d}y)$ is the space-time white noise correpsonding to the Brownian sheet $W$. Conversely,
given any sequence of independent standard Brownian motions $\beta=\{\beta_i\}_{i\in\mathbb{N}}$, the following random field
\begin{align}
	W(t,x)=\sum_{i=1}^{\infty} \beta_i(t)\int_0^x e_i(y)\mathrm{d}y,\quad (t,x)\in [0,T]\times\mathbb{R}.
\end{align}
is a Brownian sheet, and the series above converges in $L^2(\Omega)$ for each $(t,x)$.
By the similar arguments as in Proposition 3 of \cite{BD2008} (see also Proposition 11.9 in \cite{BD2019}), we can deduce that there exists a measurable map $g: C([0,T], \mathbb{R}^\infty)\rightarrow C([0,T]\times\mathbb{R},\mathbb{R})$ such that $W=g(\beta)$ a.s., where $\beta=\{\beta_i\}_{i\in\mathbb{N}}$ is defined by (\ref{230413.1626}).

By the Yamada-Watanabe theorem and Proposition \ref{wellposedeqn1},  there exists a measurable mapping $\mathcal{G}^\epsilon:C([0,T]\times\mathbb{R},\mathbb{R})\to C([0,T],C_{tem})$ such that $\mathcal{G}^\epsilon(W)=u^\epsilon$, which is the unique solution of \eqref{eqn1}. Applying the Girsanov theorem (see Theorem 1.6 in \cite{MN} or Theorem 10.14 in \cite{DZ2014}),
for any $N>0$ and $h_\epsilon\in \mathcal{X}_N$, the random field
\begin{align}
	\widetilde W(t,x):= W(t,x) + \frac{1}{\sqrt{\epsilon}}\int_0^t\int_0^x h_{\epsilon}(s,y)\mathrm{d}s\mathrm{d}y
\end{align}
is a Brownian sheet under a new probability measure, which is mutually absolute continuity with respect to $\mathbb{P}$.
%
%
Therefore, $X^{\epsilon,h_\epsilon}:=\mathcal{G}^\epsilon(\widetilde W)$ satisfies  the following stochastic equation $\mathbb{P}$-a.s.
\begin{eqnarray}\label{eqnhe}
	X^{\epsilon,h_\epsilon}(t,x)\!\!\!\!&=&\!\!\!\!P_tu_0(x)+\int_{0}^{t}\int_{\mathbb{R}}p_{t-s}(x,y)b(X^{\epsilon,h_\epsilon}(s,y))\mathrm{d}s\mathrm{d}y\nonumber\\
            &&+
            \sqrt{\epsilon}\int_{0}^{t}\int_{\mathbb{R}}p_{t-s}(x,y)\sigma(X^{\epsilon,h_\epsilon}(s,y))W(\mathrm{d}s,\mathrm{d}y) \nonumber\\
	&&+
  \int_{0}^{t}\int_{\mathbb{R}}p_{t-s}(x,y)\sigma(X^{\epsilon,h_\epsilon}(s,y))h_\epsilon(s,y)\mathrm{d}s\mathrm{d}y.
\end{eqnarray}

Using the  arguments similar to those in \cite{BD2008} (see also Chapter 11 in \cite{BD2019})
and Theorem 3.2 in \cite{MSZ}, Theorem \ref{221107.1426} is established once we have proved the following two claims:

(\textbf{C1}) For any $N>0$, $\{h_n\}_{n\in\mathbb{N}}\subset \mathcal{H}_N$ and $h\in\mathcal{H}_N$ with $h_n\to h$ weakly in $\mathcal{H}$ as $n\to\infty$, then $\lim_{n\rightarrow\infty}\mathcal{G}^0(\text{Int}(h_n))=\mathcal{G}^0(\text{Int}(h))$ in $C([0,T],C_{tem})$, where the mapping $\text{Int}$ is defined by (\ref{221105.2016}).
%

(\textbf{C2}) For any $N>0$, $\{h_\epsilon,\epsilon>0\}\subset \mathcal{X}_N$ and $\delta>0$,
\begin{align*}
	\lim_{\epsilon\to 0} \mathbb{P} \left( \tilde{d}\left(X^{\epsilon,h_\epsilon},Y^{h_\epsilon}\right)>\delta  \right)=0,
\end{align*}
where $Y^{h_\epsilon}:=\mathcal{G}^{0}(\text{Int}(h_{\epsilon}))$.

The proof of Claims {\rm(\textbf{C1})} and {\rm(\textbf{C2})} will be divided into two cases. For Case 1, we assume that $u_0\in C_{tem}$, {\rm(\textbf{H1(b)})} and {\rm(\textbf{H1(c)})} hold. The proof of Claims {\rm(\textbf{C1})} and {\rm(\textbf{C2})} is given in Propositions \ref{C1} and \ref{C2} of Section \ref{221107.1429}, respectively. For Case 2, we assume that $u_0\in C_{tem}$ and {\rm(\textbf{H0(b)})} holds. The proof of Claims {\rm(\textbf{C1})} and {\rm(\textbf{C2})} is given in
 Propositions \ref{C1 Case 2} and \ref{C2 Case 2} of Section \ref{Sec 5}, respectively.

The outline of the proof of Theorem \ref{221107.1426} is complete.
\end{proof}

\section{Well-posedness for the skeleton equations (\ref{ske0})}\label{221107.1422}
\setcounter{equation}{0}

This section is devoted to establish the existence and uniqueness of solution to the skeleton equation (\ref{ske0}); see Propositions \ref{Propo well posed SE 01} and \ref{Propo well posed SE 02}.
\vskip 0.2cm

By a standard fixed point argument, we have
\begin{proposition}\label{Propo well posed SE 01}
Assume that $u_0\in C_{tem}$ and {\rm(\textbf{H0(b)})}  holds. For any $h\in\mathcal{H}$, there exists a unique solution $Y^h$ to (\ref{ske0}). Moreover, $Y^h\in C([0,T],C_{tem})$.
\end{proposition}

In the following, we will prove that
\begin{proposition}\label{Propo well posed SE 02}
Let $u_0\in C_{tem}$. Assume that {\rm(\textbf{H1(b)})} and {\rm(\textbf{H1(c)})} hold.
 For any $h\in\mathcal{H}$, there exists a unique solution $Y^h$ to (\ref{ske0}). Moreover, $Y^h\in C([0,T],C_{tem})$.
\end{proposition}
The proof of the above proposition is inspired by \cite{SZ}, and it is divided into two parts: Existence and Uniqueness; see Propositions \ref{221105.1302} and  \ref{221105.1303}, respectively.


\subsection{Existence}

In this subsection, we assume that

\textbf{Hypothesis 2:} (\textbf{H1(a)}) holds, $\sigma$ is continuous and $\sup_{u\in\mathbb{R}}|\sigma(u)|\le K_\sigma<\infty$.
\vskip 0.2cm

The main result of this subsection is Proposition \ref{221105.1302}, in which, under \textbf{Hypothesis 2} which is weaker than that of Proposition \ref{Propo well posed SE 02},  we prove the existence of a solution $Y^h$ to (\ref{ske0}).
To this end, we approximate the coefficients $b$ and $\sigma$ by Lipschitz continuous functions $b_n$ and $\sigma_n$ (see (\ref{eq bn}) and (\ref{eq sigma n})). And then we introduce auxiliary equations (\ref{ske1}), for which, by Proposition \ref{Propo well posed SE 01}, there exist unique solutions denoted by $Y^n, n\in\mathbb{N}$.  Then we prove the pre-compactness of $\{Y^n,n\in\mathbb{N}\}$ in $C([0,T],C_{tem})$. Finally, taking $Y$ as a limit of $\{Y^n,n\in\mathbb{N}\}$, we  identify that $Y$ is a solution to (\ref{ske0}).
To prove the second statement, a compactness criterion of a subset of $C([0,T],C_{tem})$ and some a priori estimates for $Y^n$, $U_n$ and $V_n$ play an important role; see Proposition \ref{criteria} and Lemmas \ref{u_nbdd} and \ref{UnVn}, respectively.

\vspace{0.2cm}
Now, we introduce auxiliary equations (\ref{ske1}). Let $\varphi$ be a nonnegative symmetric smooth function on $\mathbb{R}$ with $\mathrm{supp}\varphi\subset (-1,1)$ and $\int_{\mathbb{R}} \varphi(x) \mathrm{d} x=1$. Let $\left\{\eta_{n}\right\}_{n\in\mathbb{N}}$ be a sequence of symmetric smooth functions such that, for each $n \geq 1,0 \leq \eta_{n} \leq 1, \eta_{n}(x)=1$
if $|x| \leq n$, and $\eta_{n}(x)=0$ if $|x| \geq n+2$. Set
\begin{align}
		b_{n}(x) &:=n \int_{\mathbb{R}} b(y) \varphi(n(x-y)) \mathrm{d} y \,\eta_{n}(x), \label{eq bn}\\
		\sigma_{n}(x) &:=n \int_{\mathbb{R}} \sigma(y) \varphi(n(x-y)) \mathrm{d} y \, \eta_{n}(x).\label{eq sigma n}
\end{align}
It is not difficult  to check that there exist nonnegative constants $L_{n}, L_{b}$ and $K_{n}$ such that, for any $x, y \in \mathbb{R}$,
\begin{align}
		\left|b_{n}(x)-b_{n}(y)\right| & \leq L_{n}|x-y|, \label{bn1}\\
		\left|b_{n}(x)\right| & \leq c_{1}|x| \log_{+}|x|+L_{b}(|x|+1),
		\label{bn2} \\
		\left|\sigma_{n}(x)-\sigma_{n}(y)\right| & \leq K_{n}|x-y|, \label{sigman1} \\
		\left|\sigma_{n}(x)\right| & \leq K_{\sigma},\label{sigman2}
\end{align}
where $c_{1}$ is the nonnegative constant appearing in (\textbf{H1(a)}), and $L_{b}$ is independent of $n$. Moreover,
\begin{align}\label{bnsigmanconv}
	\text{if } \lim_{n\rightarrow\infty}x_{n}= x\text{ in }\mathbb{R},\text{ then } \lim_{n\rightarrow\infty}b_{n}\left(x_{n}\right) = b(x)  \text{ and }
	\lim_{n\rightarrow\infty}\sigma_{n}\left(x_{n}\right)  = \sigma(x).
\end{align}

For any fixed $h\in\mathcal{H}$ and each $n\geq 1$, we consider the following auxiliary equation,
\begin{align}\label{ske1}
		Y^{n}(t,x)=& P_{t} u_{0}(x)+\int_{0}^{t} \int_{\mathbb{R}} p_{t-s}(x, y) b_{n}\left(Y^{n}(s, y)\right) \mathrm{d}s \mathrm{d} y \nonumber\\
		&+\int_{0}^{t} \int_{\mathbb{R}} p_{t-s}(x,y) \sigma_{n}\left(Y^{n}(s, y)\right) h(s,y)\mathrm{d}s\mathrm{d}y,\quad t\in[0,T],x\in \mathbb{R}.
\end{align}
Since both $b_n$ and $\sigma_n$ are globally Lipschitz, by Proposition \ref{Propo well posed SE 01}, there exists a unique solution $Y^n\in C([0,T],C_{tem})$ to (\ref{ske1}).

\vspace{0.2cm}
Before proving the pre-compactness of $\{Y^n,n\in\mathbb{N}\}$ in $C([0,T],C_{tem})$,
we first present a compactness criterion of a subset of $C([0,T],C_{tem})$, which is a modification of Lemma 6.3 in \cite{Shiga}.
\begin{proposition}
\label{criteria}
	Given a sequence $\{\phi_n\}_{n\ge 1}$ in $C([0,T],C_{tem})$, if it satisfies
\begin{itemize}
	\item [(i)] For any $\lambda>0$, $\sup_{n\ge 1}\sup_{t\leq T, x\in\mathbb{R}}\left(|\phi_n(t,x)|e^{-\lambda |x|}\right)<\infty.$
	\item [(ii)] For any $k\in\mathbb{N}$, there exist constants $a_1,a_2,a_3>0$ dependent on $k$ such that
	\begin{align*}
		\sup_{n\geq1}|\phi_n(t,x)-\phi_n(s,y)| \leq a_3 \left(|t-s|^{a_1}+|x-y|^{a_2}\right),
	\end{align*} for any $t,s\in[0,T]$ and any $x,y\in J_k:=[-k,k]$.
\end{itemize}
Then $\{\phi_n\}_{n\ge 1}$ is pre-compact in $C([0,T],C_{tem})$.
\end{proposition}

\begin{proof}
	According to the Arzel\`{a}-Ascoli theorem,
by the diagonal argument, there exist a subsequence $\{\tilde{\phi}_n\}\subset \{\phi_n\}_{n\ge 1}$ and $\phi\in C([0,T]\times\mathbb{R},\mathbb{R})$ such that,
 for every $k\in\mathbb{N}$,
	\begin{align}\label{221105.1345}
		\sup_{t\leq T, x\in J_k} |\tilde{\phi}_{n}(t,x)-\phi(t,x)|\to0,\text{ as } n\to \infty.
	\end{align}
	Now we prove that $\{\tilde{\phi}_{n}\}$ converges to $\phi$ in $C([0,T],C_{tem})$. For any $\lambda>0$,
	\begin{align*}
		&\sup_{t\leq T, x\in\mathbb{R}} \left(|\tilde{\phi}_{n}(t,x)-\tilde{\phi}_{m}(t,x)|e^{-\lambda |x|}\right)\\
		\leq &\sup_{t\leq T,x\in J_k} \left(|\tilde{\phi}_{n}(t,x)-\tilde{\phi}_{m}(t,x)|e^{-\lambda |x|}\right)+\sup_{t\leq T,x\in J_k^c} \left(|\tilde{\phi}_{n}(t,x)-\tilde{\phi}_{m}(t,x)|e^{-\lambda |x|}\right)\\
		\leq &\sup_{t\leq T,x\in J_k} \left(|\tilde{\phi}_{n}(t,x)-\tilde{\phi}_{m}(t,x)|e^{-\lambda |x|}\right)+\sup_{t\leq T,x\in J_k^c} \left(|\tilde{\phi}_{n}(t,x)|e^{-\frac{\lambda}{2} |x|}+|\tilde{\phi}_{m}(t,x)|e^{-\frac{\lambda}{2} |x|} \right)e^{-\frac{\lambda}{2} |x|}\\
	 	\leq &\sup_{t\leq T,x\in J_k} \left(|\tilde{\phi}_{n}(t,x)-\tilde{\phi}_{m}(t,x)|e^{-\lambda |x|}\right)+C\sup_{|x|\ge k} e^{-\frac{\lambda}{2} |x|}.
	\end{align*}
First letting $m,n\to\infty$ and then letting $k\to\infty$ yield that
\begin{align*}
\sup_{t\leq T, x\in\mathbb{R}} \left(|\tilde{\phi}_{n}(t,x)-\tilde{\phi}_{m}(t,x)|e^{-\lambda |x|}\right)\to 0.
\end{align*}
It follows that $\{\tilde \phi_{n}\}$ is a Cauchy sequence in $C([0,T],C_{tem})$, and its limit is denoted by $\tilde \phi$. In view of (\ref{221105.1345}), $\tilde \phi=\phi$, completing the proof of this proposition.
\end{proof}

\vspace{0.8cm}
We will establish some uniform estimates of the solutions $\{Y^{n}\}_{n\ge 1}$ to the equation (\ref{ske1}). To do this, we first introduce some notations which will be used frequently in the following. For $\lambda,\ \kappa>0$, set
\begin{align}
	\beta(\kappa,\lambda)&:=\max\left\{\frac{\lambda^2}{2},4\kappa\right\},\label{eq beta}\\
	T^*(\kappa,\lambda)&:=\frac{1}{2\beta(\kappa,\lambda)}\left[1+\log\left(\frac{4\beta(\kappa,\lambda)}{\lambda^2}\log\frac{\beta(\kappa,\lambda)}{2\kappa}\right)\right].\label{T}
\end{align}
The definition of $\beta$ implies that \eqref{T} is well-defined, and it is not difficult to prove that
\begin{eqnarray}\label{eq T}
\text{for any }\kappa>0,\ \lim_{\lambda\to 0}T^*(\kappa,\lambda)=\infty.
\end{eqnarray}

\begin{lemma}\label{u_nbdd}
	Assume $u_{0} \in C_{tem}$ and \textbf{Hypothesis 2} holds.  Then for any $\lambda,T>0$, we have
	\begin{align}
		\sup_{n \geq 1} \sup_{t \leq T, x \in \mathbb{R}}\left(\left|Y^{n}(t, x)\right| e^{-\lambda|x|e^{\beta t}}\right)\leq C_{c_1,\lambda,u_0,K_\sigma,L_b,|h|_{\mathcal{H}},T},
	\end{align}
where $c_1$ is the constant appearing in {\rm(\textbf{H1(a)})}  and we write $\beta\left(c_{1},\lambda\right)$ as $\beta$ for simplicity.

Since the constant $C_{c_1,\lambda,u_0,K_\sigma,L_b,|h|_{\mathcal{H}},T}$ is increasing w.r.t. $|h|_{\mathcal{H}}$(not dependent on $h$),
 for any $\lambda,T>0$ and $N>0$,
	\begin{align}
		\sup_{h\in\mathcal{H}_N}\sup_{n \geq 1} \sup_{t \leq T, x \in \mathbb{R}}\left(\left|Y^{n}(t, x)\right| e^{-\lambda|x|e^{\beta t}}\right)\leq C_{c_1,\lambda,u_0,K_\sigma,L_b,N,T}.
	\end{align}

\end{lemma}

\begin{proof}
	 For any $s\geq0$, let
	\begin{align}
		f_n(s):=\sup_{t\le s,x\in\mathbb{R}}\left(|Y^{n}(t,x)|e^{-\lambda |x|e^{\beta t}}\right).
	\end{align}
	Fix $T>0$. By equation (\ref{ske1}) we have
	\begin{align}\label{fn}
			f_{n}(T) \leq & \sup_{t \leq T, x \in \mathbb{R}}\left\{\left|P_{t} u_{0}(x)\right| e^{-\lambda |x|e^{\beta t}}\right\} \nonumber\\
			&+\sup_{t \leq T, x \in \mathbb{R}}\left\{\left|\int_{0}^{t} \int_{\mathbb{R}} p_{t-s}(x, y) b_{n}\left(Y^{n}(s, y)\right) \mathrm{d}s\mathrm{d} y\right| e^{-\lambda |x| e^{\beta t}}\right\} \nonumber\\
			&+\sup _{t \leq T, x \in \mathbb{R}}\left\{\left|\int_{0}^{t} \int_{\mathbb{R}} p_{t-s}(x, y) \sigma_{n}\left(Y^{n}(s, y)\right) h(s,y)\mathrm{d} s\mathrm{d} y\right| e^{-\lambda |x| e^{\beta t}}\right\}.
	\end{align}
	For the first term on the right-hand side of the above inequality, we deduce that
	\begin{align}\label{fn1}
		&\sup_{t \leq T, x \in \mathbb{R}}\left\{\left|P_{t} u_{0}(x)\right| e^{-\lambda |x| e^{\beta t}}\right\} \nonumber\\
		\leq & \sup _{t \leq T, x \in \mathbb{R}}\left\{\left|\int_{\mathbb{R}} p_{t}(x, y) u_{0}(y) \mathrm{d} y\right| e^{-\lambda |x|}\right\} \nonumber\\
		\leq & \sup_{y \in \mathbb{R}}\left(\left|u_{0}(y)\right| e^{-\lambda |y|}\right)  \sup _{t \leq T, x \in \mathbb{R}}\left\{\int_{\mathbb{R}} p_{t}(x, y) e^{\lambda |y|} \mathrm{d} y  e^{-\lambda |x|}\right\} \nonumber\\
		\leq & 2 e^{\frac{\lambda^{2} T}{2}} \left|u_{0}\right|_{(-\lambda)},
	\end{align}
	where in the last step we have used (\ref{3.20}). Applying H\"older's inequality, \eqref{sigman2} and (\ref{3.21})
	to the third term on the right-hand side of (\ref{fn}), we find
	\begin{align}
		\label{fn3}
			&\sup _{t \leq T, x \in \mathbb{R}}\left\{\left|\int_{0}^{t} \int_{\mathbb{R}} p_{t-s}(x, y) \sigma_{n}\left(Y^{n}(s, y)\right) h(s,y)\mathrm{d} s\mathrm{d} y\right| e^{-\lambda |x| e^{\beta t}}\right\} \nonumber\\
			\leq& K_\sigma \left|h\right|_{\mathcal{H}} \sup _{t \leq T, x \in \mathbb{R}}\left(\int_{0}^{t} \int_{\mathbb{R}} p_{t-s}^2(x,y)\mathrm{d}s\mathrm{d}y \right)^{\frac{1}{2}} \nonumber\\
			\leq &  K_\sigma \left|h\right|_{\mathcal{H}} \sup _{t \leq T} \left(\int_{0}^{t} \frac{1}{\sqrt{\pi(t-s)}}\mathrm{d}s\right)^{\frac{1}{2}}\nonumber\\
			\leq& \sqrt{2}\pi^{-\frac{1}{4}}K_\sigma \left|h\right|_{\mathcal{H}} T^\frac{1}{4}.
	\end{align}
Let us estimate the second term on the right-hand side of (\ref{fn}). By (\ref{bn2}) and $\log_+(ab)\leq \log_+a+\log_+b$ for any $a,b\ge 0$,
	\begin{align}\label{fn2}
	&\sup _{t \leq T, x \in \mathbb{R}}\left\{\left|\int_{0}^{t} \int_{\mathbb{R}} p_{t-s}(x, y) b_{n}\left(Y^{n}(s, y)\right) \mathrm{d}s \mathrm{d} y\right| e^{-\lambda |x| e^{\beta t}}\right\} \nonumber\\
	\leq & c_{1}\sup _{t \leq T, x \in \mathbb{R}}\left\{\int_{0}^{ t } \int_{\mathbb{R} } p_{ t-s } (x , y) \left|Y^{n}(s, y)\right| \log _{+}\left|Y^{n}(s,y)\right|\mathrm{d}s\mathrm{d}ye^{-\lambda |x| e^{\beta t}}\right\} \nonumber\\
	&+L_{b}\sup _{t \leq T, x \in \mathbb{R}}\left\{ \int_{0}^{ t } \int_{\mathbb{R} } p_{t-s} (x,y)  \left|Y^{n}(s, y)\right|\mathrm{d} s \mathrm{d} ye^{-\lambda |x| e^{\beta t}}\right\} +L_{b} T \nonumber\\
	\leq &L_bT+L_b \sup _{t \leq T, x \in \mathbb{R}} \left\{ \int_{0}^{t}\int_{\mathbb{R} } p_{t-s}(x,y)e^{\lambda |y| e^{\beta s}}\left|Y^{n}(s, y)\right|e^{-\lambda |y| e^{\beta s}}\mathrm{d}s\mathrm{d} ye^{-\lambda |x| e^{\beta t}}\right\}\nonumber\\
	&+c_{1} \sup _{t\leq T, x\in \mathbb{R}}\left\{\int_{0}^{t} \int_{\mathbb{R}} p_{t-s}(x,y)e^{\lambda |y| e^{\beta s}}\left|Y^{n}(s, y)\right| e^{-\lambda |y| e^{\beta s}}\right. \nonumber\\
	&\qquad\left.\times \log_{+}\left(|Y^{n}(s,y)|e^{-\lambda |y| e^{\beta s}}e^{\lambda |y| e^{\beta s}} \right) \mathrm{d}s\mathrm{d}y e^{-\lambda|x| e^{\beta t}}\right\} \nonumber\\		
	\leq &L_bT+L_b \sup _{t\leq T,x\in\mathbb{R}}\left\{ \int_{0}^{t} \sup_{y\in \mathbb{R}} \left(\left|Y^{n}(s, y)\right|e^{-\lambda |y| e^{\beta s}}\right) \int_{\mathbb{R} } p_{t-s}(x,y)e^{\lambda |y| e^{\beta s}}\mathrm{d} y \mathrm{d} se^{-\lambda |x| e^{\beta t}}\right\} \nonumber\\
	&+c_{1} \sup _{t \leq T, x \in \mathbb{R}}\left\{\int_{0}^{t} \sup _{y \in \mathbb{R}}\left[\left(|Y^{n}(s, y)| e^{-\lambda |y| e^{\beta s}}\right) \log_{+}\left(|Y^{n}(s, y)| e^{-\lambda |y| e^{\beta s}}\right)\right]\right. \nonumber\\
	&\qquad \left.\times \int_{\mathbb{R}} p_{t-s}(x,y) e^{\lambda |y| e^{\beta s}} \mathrm{d}y \mathrm{d} s \,e^{-\lambda|x| e^{\beta t}}\right\} \nonumber\\
	&+c_1\sup_{t \leq T, x \in \mathbb{R}} \left\{ \int_{0}^{t} \sup _{y \in \mathbb{R}}\left(|Y^{n}(s, y)| e^{-\lambda |y| e^{\beta s}}\right) \int_{\mathbb{R}}  p_{t-s}(x,y) e^{\lambda |y| e^{\beta s}} \lambda |y| e^{\beta s}\mathrm{d}y \mathrm{d}s \,e^{-\lambda |x| e^{\beta t}}\right\} \nonumber \\
	=:&L_bT+\uppercase\expandafter{\romannumeral1}+\uppercase\expandafter{\romannumeral2}+\uppercase\expandafter{\romannumeral3}.
\end{align}
In the following we estimate $\uppercase\expandafter{\romannumeral1},\ \uppercase\expandafter{\romannumeral2}$ and $ \uppercase\expandafter{\romannumeral3}$ separately. By (\ref{3.20}), for $0\leq s\leq t$,
\begin{align}\label{4.37}
	\int_{\mathbb{R}} p_{t-s}(x,y)e^{\lambda |y|e^{\beta s}}\mathrm{d}y
	\leq 2e^{\frac{\lambda ^{2}(t-s) e^{2 \beta s}}{2}}e^{\lambda |x|e^{\beta s}}
	\leq 2 e^{\frac{\lambda^2}{4\beta}e^{2\beta t-1}} e^{\lambda |x|e^{\beta t}}.
\end{align}
It follows by the above inequality that
	\begin{align}\label{fn21}
			\uppercase\expandafter{\romannumeral1}\leq 2 L_b e^{\frac{\lambda^{2}}{4 \beta}e^{2\beta T-1}} \sup_{t \leq T}\left\{\int_{0}^{t} \sup_{y \in \mathbb{R}} \left(\left|Y^{n}(s,y)\right|e^{-\lambda |y|e^{\beta s}}\right)\mathrm{d}s \right\}
			\leq 2 L_b e^{\frac{\lambda ^{2}}{4 \beta} e^{2 \beta T-1}} \int_{0}^{T} f_n(s)\mathrm{d}s.
	\end{align}
	Notice that the function $x\mapsto x\log_{+}x$ is increasing on $[0,\infty)$, by \eqref{4.37} again,
	\begin{align}\label{fn22}
			\uppercase\expandafter{\romannumeral2}\leq & 2 c_{1} e^{\frac{\lambda ^{2}}{4 \beta} e^{2 \beta T-1}} \int_{0}^{T} f_n(s) \log_{+} f_n(s) \mathrm{d} s.
	\end{align}
	By (\ref{3.22}), for $0\leq s\leq t$,
	\begin{align}\label{4.3}
		\int_{\mathbb{R}} p_{t-s}(x,y)e^{\lambda |y|e^{\beta s}}\lambda |y|e^{\beta s}\mathrm{d}y
		\leq &e^{\frac{\lambda ^{2}(t-s) e^{2 \beta s}}{2}} e^{\lambda |x| e^{\beta s}} \lambda |x| e^{\beta s}+C_{\lambda , \beta, t} e^{\lambda |x| e^{\beta s}},
	\end{align}
where the constant $C_{\lambda ,\beta,t}$ is increasing with respect to $t$. Using the above inequality, we deduce that
	\begin{align*}
			\uppercase\expandafter{\romannumeral3}\le & c_{1} \sup _{t \leq T, x \in \mathbb{R}}\left\{\int_{0}^{t} \sup _{y \in \mathbb{R}}\left(|Y^{n}(s, y)| e^{-\lambda |y| e^{\beta s}}\right)\right. \nonumber \\
			&\qquad\left.\times\left(e^{\frac{\lambda ^{2}(t-s) e^{2 \beta s}}{2}} e^{\lambda |x| e^{\beta s}} \lambda |x| e^{\beta s}+C_{\lambda , \beta, t} e^{\lambda |x| e^{\beta s}}\right) \mathrm{d} s \,e^{-\lambda |x| e^{\beta t}}\right\} \nonumber\\
			\leq & c_{1} \sup_{t \leq T, x \in \mathbb{R}}\left\{\sup_{s \leq t, y \in \mathbb{R}}\left(|Y^{n}(s, y)| e^{-\lambda |y| e^{\beta s}}\right) \frac{1}{\beta} e^{\frac{\lambda ^{2}}{4 \beta} e^{2 \beta t-1}} \int_{0}^{t} \left(e^{\lambda |x| e^{\beta s}}\right)' \mathrm{d} s \,e^{-\lambda |x| e^{\beta t}}\right\} \nonumber \\
			&+c_{1} \sup_{t \leq T}\left\{C_{\lambda ,\beta,t} \int_{0}^{t} \sup_{r \leq s, y \in \mathbb{R}}\left(|Y^{n}(r, y)| e^{-\lambda |y| e^{\beta r}}\right) \mathrm{d} s\right\} \nonumber\\
			\le& \frac{c_{1}}{\beta} e^{\frac{\lambda ^{2}}{4 \beta} e^{2 \beta T-1}} f_n(T)+C_{c_1, \lambda ,\beta,T} \int_{0}^{T} f_n(s) \mathrm{d}s.
	\end{align*}
Combining the above estimate with the following two facts that

(F1) (\ref{eq T}) implies that there exists a positive constant $\lambda_{T}$ such that $T \leq T^{*}\left( c_{1},\lambda\right)$ for all $\lambda \in(0, \lambda_{T}]$,

(F2)
	\begin{align}\label{T*}
		\frac{c_{1}}{\beta} e^{\frac{\lambda ^{2}}{4 \beta} e^{2 \beta T-1}} \leq \frac{1}{2} \Longleftrightarrow T \leq T^{*}\left( c_{1},\lambda \right)=\frac{1}{2 \beta}\left[1+\log \left(\frac{4 \beta}{\lambda^{2}} \log \frac{\beta}{2 c_{1}}\right)\right],
	\end{align}
we see that, for any $\lambda \in(0, \lambda_{T}]$,
	\begin{align}\label{fn23}
		\uppercase\expandafter{\romannumeral3}\leq \frac{1}{2} f_n(T)+C_{c_1,\lambda ,\beta,T} \int_{0}^{T} f_n(s)\mathrm{d}s.
	\end{align}
Substituting (\ref{fn1}), (\ref{fn3}), (\ref{fn2}), (\ref{fn21}), (\ref{fn22}),  (\ref{fn23}) into (\ref{fn}) leads to, for any $\lambda \in(0, \lambda_{T}]$,
	\begin{align*}
			f_n(T)\le& 4 e^{\frac{\lambda ^{2} T}{2}} \left|u_{0}\right|_{(-\lambda )} +2\sqrt{2}\pi^{-\frac{1}{4}} K_\sigma \left|h\right|_{\mathcal{H}} T^\frac{1}{4}+2L_bT \nonumber\\
			&+\left(4 L_b e^{\frac{\lambda ^{2}}{4 \beta} e^{2 \beta T-1}}+2C_{c_{1},\lambda ,\beta,T}\right) \int_{0}^{T} f_n(s)\mathrm{d}s+4 c_{1} e^{\frac{\lambda ^{2}}{4 \beta} e^{2 \beta T-1}} \int_{0}^{T} f_n(s) \log _{+} f_n(s) \mathrm{d}s,
	\end{align*}
	that is,
	\begin{align*}
		f_n(T)\le C_{\lambda ,u_0,K_\sigma,L_b,|h|_{\mathcal{H}},T}+
		C_{c_1,\lambda ,\beta,L_b,T}\int_{0}^{T} f_n(s)\mathrm{d}s+C_{c_1,\lambda ,\beta,T}\int_{0}^{T} f_n(s) \log _{+} f_n(s) \mathrm{d}s.
	\end{align*}
	It follows from Lemma \ref{Gronwall1} and the above inequality that,  for any $\lambda \in(0, \lambda_{T}]$,
	\begin{align}\label{u_nbddbeta}
		\sup_{t\le  T,x\in\mathbb{R}}\left(|Y^{n}(t,x)|e^{-\lambda |x|e^{\beta t}}\right)\leq C_{c_1,\lambda ,\beta,u_0,K_\sigma,L_b,|h|_{\mathcal{H}},T},\quad \forall\  n\ge 1.
	\end{align}
Since the constant $C_{c_1,\lambda ,\beta,u_0,K_\sigma,L_b,|h|_{\mathcal{H}},T}$ is independent on $n$, we may rewrite (\ref{u_nbddbeta}) as,  for any $\lambda \in(0, \lambda_{T}]$,
\begin{align*}
	\sup_{n\ge 1}\sup_{t\le T,x\in\mathbb{R}}\left(|Y^{n}(t,x)|e^{-\lambda|x|e^{\beta t}}\right)\leq C_{c_1,\lambda,u_0,K_\sigma,L_b,|h|_{\mathcal{H}},T}.
	\end{align*}
Combining the fact that it suffices to prove this lemma for sufficiently small $\lambda$, the proof of this lemma is complete.

\end{proof}

\begin{lemma}\label{UnVn}
	For each $n\ge 1$, define
	\begin{align}
		U_{n}(t, x):&=\int_{0}^{t} \int_{\mathbb{R}} p_{t-r}(x, z) b_{n}\left(Y^{n}(r, z)\right) \mathrm{d} r \mathrm{d} z,\\
		V_{n}(t, x):&=\int_{0}^{t} \int_{\mathbb{R}} p_{t-r}(x, z) \sigma_{n}\left(Y^{n}(r, z)\right)h(r,y) \mathrm{d}r \mathrm{d}z,
	\end{align}
	where $h\in\mathcal{H}$. Assume that $u_{0} \in C_{tem}$and \textbf{Hypothesis 2} holds. Then for any $\lambda, T>0$, $\theta \in(0,1)$ and $k\in\mathbb{N}$, there exist constants $C_{c_{1},\lambda,u_{0},K_{\sigma},L_{b},|h|_{\mathcal{H}},T,k,\theta}$ and $C_{K_{\sigma}, |h|_{\mathcal{H}}}$ such that, for any $s, t \in[0, T]$ and $x, y \in [-k,k]$,
	\begin{align}
\sup_{n\in\mathbb{N}}\left|U_{n}(t, x)-U_{n}(s, y)\right| &\leq C_{c_{1},\lambda,u_{0},K_{\sigma},L_{b},|h|_{\mathcal{H}},T,k, \theta}\left(|t-s|^{\theta }+|x-y|\right), \label{U_n}\\
\sup_{n\in\mathbb{N}}\left|V_{n}(t, x)-V_{n}(s, y)\right| &\leq C_{K_{\sigma}, |h|_{\mathcal{H}}}\left(|t-s|^{\frac{1}{4}}+|x-y|^{\frac{1}{2}}\right). \label{V_n}
	\end{align}

Since the constants appearing in the above two inequalities are  increasing w.r.t. $|h|_{\mathcal{H}}$(not dependent on $h$),
 for any $\lambda,T>0$, $\theta \in(0,1)$, $k\in\mathbb{N}$, $N>0$, $s, t \in[0, T]$ and $x, y \in [-k,k]$,
 	\begin{align}
\sup_{h\in\mathcal{H}_N}\sup_{n\in\mathbb{N}}\left|U_{n}(t, x)-U_{n}(s, y)\right| &\leq C_{c_{1},\lambda,u_{0},K_{\sigma},L_{b},N,T,k, \theta}\left(|t-s|^{\theta }+|x-y|\right), \label{U_n-01}\\
\sup_{h\in\mathcal{H}_N}\sup_{n\in\mathbb{N}}\left|V_{n}(t, x)-V_{n}(s, y)\right| &\leq C_{K_{\sigma}, N}\left(|t-s|^{\frac{1}{4}}+|x-y|^{\frac{1}{2}}\right). \label{V_n-01}
	\end{align}
\end{lemma}

\begin{proof}
	It suffices to prove this lemma for sufficiently small $\lambda$. Fix $T>0$. Recall $\lambda_{T}$ appearing in the proof of Lemma \ref{u_nbdd}; see the fact (F1). In the following, we will prove that this lemma holds for any $\lambda \in(0, \lambda_{T}]$.


Fix $\lambda \in(0, \lambda_{T}]$. In the following we write $\beta\left(c_1,\lambda \right)$ as $\beta$ for simplicity. Without loss of generality, we assume $t \geq s$. Observe that
	\begin{align}\label{Un}
			&|U_{n}(t,x)-U_{n}(s, y)| \nonumber\\
			\leq &\left|\int_{0}^{s} \int_{\mathbb{R}}\left[p_{t-r}(x, z)-p_{s-r}(x, z)\right] b_{n}\left(Y^{n}(r, z)\right) \mathrm{d} r \mathrm{d} z\right| \nonumber\\
			&+\left|\int_{s}^{t} \int_{\mathbb{R}} p_{t-r}(x, z) b_{n}\left(Y^{n}(r,z)\right) \mathrm{d} r \mathrm{d} z\right| \nonumber \\
			&+\left|\int_{0}^{s} \int_{\mathbb{R}}\left[p_{s-r}(x, z)-p_{s-r}(y, z)\right] b_{n}\left(Y^{n}(r, z)\right) \mathrm{d} r \mathrm{d} z\right| \nonumber\\
			=:& I^n_{1}+I^n_{2}+I^n_{3}.
	\end{align}
	By (\ref{bn2}), (\ref{1oflem3.3}) and (\ref{3.20}), (\ref{3.22}) we have
	\begin{align*}
			I^n_{1} \leq & \int_{0}^{s} \int_{\mathbb{R}}\left|p_{t-r}(x, z)-p_{s-r}(x, z)\right| \left(c_{1}\left|Y^{n}(r, z)\right| \log _{+}\left|Y^{n}(r, z)\right|+L_{b}\left(\left|Y^{n}(r, z)\right|+1\right)\right) \mathrm{d} r \mathrm{d} z \nonumber\\
			\leq & \int_{0}^{s} \int_{\mathbb{R}} \frac{(2 \sqrt{2})^{\theta}|t-s|^{\theta}}{(s-r)^{\theta}}\left(p_{s-r}(x, z)+p_{t-r}(x, z)+p_{2(t-r)}(x, z)\right) \nonumber\\
			& \times\left\{c_{1} e^{\lambda|z| e^{\beta r}} \sup _{z \in \mathbb{R}}\left[\left|Y^{n}(r, z)\right| \,e^{-\lambda|z| e^{\beta r}} \log_+\left(\left|Y^{n}(r, z)\right| e^{-\lambda|z| e^{\beta r}}\right)\right]\right. \nonumber\\
			&+ c_{1} e^{\lambda|z| e^{\beta r}} \lambda|z| e^{\beta r} \sup_{z \in \mathbb{R}}\left(\left|Y^{n}(r, z)\right| e^{-\lambda|z| e^{\beta r}}\right) \nonumber\\
			&\left.+L_{b} e^{\lambda|z| e^{\beta r}}\sup _{z \in \mathbb{R}}\left(\left|Y^{n}(r, z)\right| e^{-\lambda|z| e^{\beta r}}\right)+L_{b}\right\} \mathrm{d} z \mathrm{d} r \nonumber\\
			\leq &(2 \sqrt{2})^{\theta}|t-s|^{\theta} \int_{0}^{s} \frac{d r}{(s-r)^{\theta}} \left\{ \sup_{r \leq T, z \in \mathbb{R}}\left(\left|Y^{n}(r, z)\right| e^{-\lambda|z| e^{\beta r}}\right) \log _{+}\left[\sup_{r \leq T, z \in \mathbb{R}}\left(\left|Y^{n}(r, z)\right| e^{-\lambda|z| e^{\beta r}}\right)\right] \right. \nonumber\\
			&\left.\times C_{c_{1},\lambda, \beta, T} \,e^{\lambda|x| e^{\beta r}} +\sup_{r \leq T, z \in \mathbb{R}}\left(\left|Y^{n}(r, z)\right| e^{-\lambda|z| e^{\beta r}}\right)\right. \nonumber\\
			& \left.\times C_{c_{1}, \lambda,\beta, L_{b}, T}\left(e^{\lambda|x| e^{\beta r}} \lambda|x| e^{\beta r}+e^{\lambda|x| e^{\beta r}}\right) +3 L_{b}\right\} \mathrm{d}r.
	\end{align*}
Hence it follows by
	\begin{align*}
		e^{\lambda|x| e^{\beta r}} \lambda|x| e^{\beta r}  \leq C_{\lambda,\beta,T,k},\quad \forall \  (r,x) \in[0, T]\times [-k,k]
	\end{align*}
that
	\begin{align}\label{Un1}
			I_{1}^n \leq &(2 \sqrt{2})^{\theta}|t-s|^{\theta} \int_{0}^{s} \frac{\mathrm{d}r}{(s-r)^{\theta}} \times C_{c_1,\lambda, \beta, L_{b}, T,k} \nonumber\\
			& \times \left\{\sup_{r \leq T, z \in \mathbb{R}}\left(\left|Y^{n}(r, z)\right| e^{-\lambda|z| e^{\beta r}}\right)\log _{+}\left[\sup _{r \leq T, z \in \mathbb{R}}\left(\left|Y^{n}(r, z)\right| e^{-\lambda|z| e^{\beta r}}\right)\right]\right.\nonumber\\
			&\left. +\sup_{r \leq T, z \in \mathbb{R}}\left(\left|Y^{n}(r, z)\right| e^{-\lambda|z| e^{\beta r}}\right) +1\right\}\nonumber\\
          \leq & C_{c_1,\lambda,\beta,u_0, K_\sigma, L_b, |h|_{\mathcal{H}},T,k,\theta}\,|t-s|^{\theta }.
	\end{align}
Here we have used  Lemma \ref{u_nbdd} to get the last inequality.

For $I_{2}^n$, a similar argument as proving (\ref{Un1}) shows that
	\begin{align}\label{Un2}
			I_{2}^n\le C_{c_1,\lambda,\beta,u_0, K_\sigma, L_b, |h|_{\mathcal{H}}, T,k}\,|t-s|.
	\end{align}
For $I^n_3$, by (\ref{bn2}) and (\ref{2oflem3.3})-(\ref{4oflem3.3}),
\begin{align*}
	I^n_{3} \leq & \int_{0}^{s} \int_{\mathbb{R}}\left|p_{s-r}(x, z)-p_{s-r}(y, z)\right| \left[c_{1}\left|Y^{n}(r, z)\right| \log _{+}\left|Y^{n}(r, z)\right|+L_{b}\left(\left|Y^{n}(r, z)\right|+1\right)\right] \mathrm{d} r \mathrm{d} z  \nonumber\\
	\leq & \int_{0}^{s} \int_{\mathbb{R}}\left|p_{s-r}(x, z)-p_{s-r}(y, z)\right| \nonumber\\
	& \times\left\{c_{1} e^{\lambda|z| e^{\beta r}} \lambda|z| e^{\beta r} \sup_{z \in \mathbb{R}}\left(\left|Y^{n}(r, z)\right| e^{-\lambda|z| e^{\beta r}}\right)\right. \nonumber\\
	&+c_{1} e^{\lambda|z| e^{\beta r}} \sup_{z \in \mathbb{R}}\left[\left|Y^{n}(r, z)\right| e^{-\lambda|z| e^{\beta r}} \log_{+}\left(\left|Y^{n}(r, z)\right| e^{-\lambda|z| e^{\beta r}}\right)\right]  \nonumber\\
	&\left.+L_{b} e^{\lambda|z| e^{\beta r}} \sup _{z \in \mathbb{R}}\left(\left|Y^{n}(r, z)\right| e^{-\lambda|z| e^{\beta r}}\right)+L_{b}\right\} \mathrm{d} z \mathrm{d} r \nonumber\\
	\leq &|x-y| \int_{0}^{s} \frac{1}{\sqrt{s-r}} \left\{\sup _{r \leq T, z \in \mathbb{R}}\left(\left|Y^{n}(r, z)\right| e^{-\lambda|z| e^{\beta r}}\right)\right. \nonumber\\
	& \times C_{c_{1},\lambda,\beta,T}\left[e^{\lambda(|x|+|x-y|) e^{\beta r}} \lambda(|x|+|x-y|) e^{\beta r}+e^{\lambda(|x|+|x-y|) e^{\beta r}}\right] \nonumber\\
	&+\sup_{r \leq T, z \in \mathbb{R}}\left(\left|Y^{n}(r, z)\right| e^{-\lambda|z| e^{\beta r}}\right) \log_{+}\left[\sup _{r \leq T, z \in \mathbb{R}}\left(\left|Y^{n}(r, z)\right| e^{-\lambda|z| e^{\beta r}}\right)\right] \nonumber\\
	&\left.\times C_{c_{1},\lambda,\beta, T}  \,e^{\lambda(|x|+|x-y|) e^{\beta r}}+\sqrt{\frac{2}{\pi}} L_{b}\right\} \mathrm{d} r .
\end{align*}
	Due to the fact that $|x-y| \leq 2k$, we have
	\begin{align*}
			e^{\lambda(|x|+|x-y|) e^{\beta r}} \lambda(|x|+|x-y|) e^{\beta r}
			\leq  C_{\lambda,\beta,T,k}, \quad \forall\  r \in[0, T], \  x,y\in [-k,k].
	\end{align*}
	Hence, applying Lemma \ref{u_nbdd} again, we see that
	\begin{align}\label{Un3}
		I^n_{3}\leq C_{c_{1}, \lambda,\beta, u_{0}, K_{\sigma}, L_{b}, |h|_{\mathcal{H}},T,k}\,|x-y|.
	\end{align}
	Combining \eqref{Un}-\eqref{Un3} together yields \eqref{U_n}.
	
	Now we consider $V_n$. Observe that
	\begin{align}\label{Vn}
			&\left|V_{n}(t, x)-V_{n}(s, y)\right| \nonumber\\
			\leq &\left|\int_{0}^{s} \int_{\mathbb{R}}\left[p_{t-r}(x, z)-p_{s-r}(x, z)\right] \sigma_{n}\left(Y^{n}(r, z)\right) h(r,z)\mathrm{d} r\mathrm{d} z\right| \nonumber\\
			&+\left|\int_{s}^{t} \int_{\mathbb{R}} p_{t-r}(x, z) \sigma_{n}\left(Y^{n}(r, z)\right) h(r,z)\mathrm{d} r\mathrm{d} z\right| \nonumber\\
			&+\left|\int_{0}^{s} \int_{\mathbb{R}}\left[p_{s-r}(x, z)-p_{s-r}(y, z)\right] \sigma_{n}\left(Y^{n}(r, z)\right) h(r,z)\mathrm{d} r\mathrm{d} z\right| \nonumber \\
			=& J^n_{1}+J^n_{2}+J^n_{3}.
	\end{align}
	Using H\"older's inequality and (\ref{5oflem3.3}) we get
	\begin{align}\label{Vn1}
			J_{1}^n & \leq K_{\sigma}\left|h\right|_{\mathcal{H}} \left(\int_{0}^{s} \int_{\mathbb{R}}\left|p_{t-r}(x, z)-p_{s-r}(x, z)\right|^{2} \mathrm{d} z \mathrm{d} r\right)^{\frac{1}{2}} \nonumber\\
			&\leq K_{\sigma}\left|h\right|_{\mathcal{H}} \left(\frac{\sqrt{2}-1}{\sqrt{\pi}} \left|t-s\right|^{\frac{1}{2}}\right)^{\frac{1}{2}} \nonumber\\
			& \leq \pi^{-\frac{1}{4}}K_{\sigma}\left|h\right|_{\mathcal{H}} |t-s|^{\frac{1}{4}}.
	\end{align}
	A similar argument as getting the above inequality implies that
	\begin{align}\label{Vn3}
			J^n_{3} 
			\leq \sqrt{2}\pi^{-\frac{1}{4}} K_{\sigma}\left|h\right|_{\mathcal{H}}|x-y|^{\frac{1}{2}}.
	\end{align}
	For the term $J^n_{2}$, by H\"older's inequality and  (\ref{3.21}) it follows that
	\begin{align}\label{Vn2}
			J^n_{2} \leq& K_{\sigma} \left|h \right|_{\mathcal{H}} \left( \int_{s}^{t} \int_{\mathbb{R}} p^2_{t-r}(x, z) \mathrm{d} z \mathrm{d} r\right)^{\frac{1}{2}} \nonumber\\
			 \leq& K_{\sigma} \left|h \right|_{ \mathcal{H}}  \left(\int_{s}^{t} \frac{1}{ \sqrt{\pi(t-r)}} \mathrm{d} r\right)^{\frac{1}{2}} \nonumber\\
			 = &\sqrt{2}\pi^{-\frac{1}{4}}K_{\sigma}\left|h \right|_{\mathcal{H}} |t-s|^{\frac{1}{4}}.
	\end{align}
Combining \eqref{Vn}-\eqref{Vn2} together yields \eqref{V_n}.

The proof of this lemma is complete.
\end{proof}

\vskip 0.5cm

\begin{proposition}\label{221105.1302}
	Assume that  $u_0\in C_{tem}$ and \textbf{Hypothesis 2} holds. Then for each $h\in\mathcal{H}$, there exists a solution $Y^h\in C([0,T],C_{tem})$ to equation \eqref{ske0}.
\end{proposition}
\begin{proof}
Recall $Y^{n}, {n \ge 1}$ is the solution to (\ref{ske1}), that is, for any $(t,x)\in[0,T]\times\mathbb{R}$,
\begin{align}\label{eq 2023 03}
		Y^{n}(t, x)=& P_{t} u_{0}(x)+\int_{0}^{t} \int_{\mathbb{R}} p_{t-r}(x, z) b_{n}\left(Y^{n}(r, z)\right) \mathrm{d}r\mathrm{d} z \nonumber\\
		&+\int_{0}^{t} \int_{\mathbb{R}} p_{t-r}(x,z) \sigma_{n}\left(Y^{n}(r, z)\right) h(r,z)\mathrm{d}r\mathrm{d}z.
\end{align}

By Proposition \ref{criteria}, Lemma \ref{u_nbdd} and Lemma \ref{UnVn}, $\{Y^{n}\}_{n \ge 1}$  is pre-compact in $C([0,T],C_{tem})$.
Hence, there exist a subsequence of $\{Y^{n}\}_{n\ge 1}$, still denoted as $\{Y^{n}\}_{n\ge 1}$, and $Y\in C([0,T],C_{tem})$ such that $\tilde{d}(Y^{n},Y)\to0$ as $n\to\infty$, i.e. for any $\lambda>0$,
\begin{align}\label{eq 2023 01}
	\lim_{n\to\infty}\sup_{t \leq T,x\in\mathbb{R}} \left(\left|Y^{n}(t,x)-Y(t,x)\right|e^{-\lambda |x|}\right)= 0.
\end{align}
Now, we verify that $Y$ is a solution to (\ref{ske0}).
\eqref{bn2} and Lemma \ref{u_nbdd} imply that for any fixed $(t,x)\in[0,T]\times\mathbb{R}$, $\big\{b_{n}(Y^{n}(r, z)),(r,z)\in[0,t]\times\mathbb{R}\big\}_{n\in\mathbb{N}}$ is uniformly integrable with respect to the measure $p_{t-r}(x,z)\mathrm{d}r\mathrm{d}z$ on $[0,t]\times\mathbb{R}$. Applying (\ref{bnsigmanconv}), (\ref{eq 2023 01}) and the Vitali convergence theorem yields
\begin{align}
	\lim_{n\to\infty}\int_{0}^{t} \int_{\mathbb{R}} p_{t-r}(x, z) b_{n}\left(Y^{n}(r, z)\right) \mathrm{d}r\mathrm{d} z = \int_{0}^{t} \int_{\mathbb{R}} p_{t-r}(x, z) b\left(Y(r, z)\right) \mathrm{d}r\mathrm{d}z.
\end{align}

By (\ref{sigman2}),
\begin{eqnarray*}
\sup_{n\in\mathbb{N}}|p_{t-r}(x,z)\sigma_n(Y^{n}(r,z))h(r,z)|\leq K_\sigma p_{t-r}(x,z)|h(r,z)|.
\end{eqnarray*}
And note that $h\in\mathcal{H}$ implies $\int_{0}^{t} \int_{\mathbb{R}}p_{t-r}(x,z)|h(r,z)|\mathrm{d}r\mathrm{d}z<\infty$. Hence applying (\ref{bnsigmanconv}) and (\ref{eq 2023 01}) again and the dominated convergence theorem, we have
%
\begin{align}\label{eq 2023 02}
		\lim_{n\to\infty}\int_{0}^{t} \int_{\mathbb{R}} p_{t-r}(x,z) \sigma_{n}\left(Y^{n}(r, z)\right) h(r,z)\mathrm{d}r\mathrm{d}z
		= \int_{0}^{t} \int_{\mathbb{R}} p_{t-r}(x,z) \sigma\left(Y(r, z)\right) h(r,z)\mathrm{d}r\mathrm{d}z.
\end{align}
Applying (\ref{eq 2023 01})--(\ref{eq 2023 02}) and letting $n\to\infty$ in (\ref{eq 2023 03}), we get, for any $(t,x)\in[0,T]\times\mathbb{R}$,
\begin{align*}
	Y(t, x)=& P_{t} u_{0}(x)+\int_{0}^{t} \int_{\mathbb{R}} p_{t-r}(x, z) b\left(Y(r, z)\right) \mathrm{d}r\mathrm{d} z \nonumber\\
	&+\int_{0}^{t} \int_{\mathbb{R}} p_{t-r}(x,z) \sigma\left(Y(r, z)\right) h(r,z)\mathrm{d}r\mathrm{d}z.
\end{align*}
Thus $Y\in C([0,T],C_{tem})$ is a solution to \eqref{ske0}.

The proof of this proposition is complete.
\end{proof}

\subsection{Uniqueness}

The main result of this subsection is Proposition \ref{221105.1303}, in which, under the same assumptions of Proposition \ref{Propo well posed SE 02},  we prove the uniqueness of a solution $Y^h$ to (\ref{ske0}), completing the proof of Proposition \ref{Propo well posed SE 02}.

We first prove the following result, which will be used later.
%
\begin{proposition}\label{moment}
	Let $\varphi$ be a nonnegative increasing function on $\mathbb{R}_+$ and $h\in L^2([0,T]\times\mathbb{R})$. Assume that $\tilde{\sigma}: [0,T]\times \mathbb{R}\rightarrow\mathbb{R}$ satisfies the Carath\'eodory conditions, i.e., $\tilde{\sigma}(s,y)$ is continuous with respect to $y\in\mathbb{R}$ for a.e. $s\in [0,T]$ and Borel measurable with respect to $s\in [0,T]$ for each $y\in\mathbb{R}$. Then for any $p,\eta,T>0$,
	\begin{align}\label{moment1}
			&\sup_{t\leq T, x\in\mathbb{R}}\left\{
			\left|\int_{0}^{t} \int_{\mathbb{R}} p_{t-s}(x,y)\tilde{\sigma}(s,y)h(s,y)\mathrm{d}s\mathrm{d}y\right|^{p}e^{-p\varphi(t)|x|} \right\} \nonumber\\
			\leq &\eta\sup_{t\leq T, x\in\mathbb{R}}\left(\left|\tilde{\sigma}(t,x)\right|e^{-\varphi(t)|x|}\right)^{p} \nonumber\\
			&+C_{\eta,T,\varphi(T),|h|_{\mathcal{H}},p} \int_{0}^{T} \left(\int_{\mathbb{R}}\left|h(s,y)\right|^2\mathrm{d}y\right)\sup_{y\in\mathbb{R}} \left(\left|\tilde{\sigma}(s,y)\right|e^{-\varphi(s)|y|}\right)^{p} \mathrm{d}s.
	\end{align}
\end{proposition}

\begin{proof}
For $p>2$, applying H\"older's inequality and (\ref{3.21}) yields
	\begin{align*}
		&\left|\int_{0}^{t} \int_{\mathbb{R}} p_{t-s}(x, y) \tilde{\sigma}(s, y) h(s, y) \mathrm{d}s\mathrm{d}y \right|\,e^{-\varphi(t)|x|}\\
		\leq& \left(\int_{0}^{t} \int_{\mathbb{R}}\left(|\tilde{\sigma}(s, y)| e^{-\varphi(s)|y|}|h(s, y)|^{\frac{2}{p}}\right)^{p} \mathrm{d}s \mathrm{d}y\right)^{\frac{1}{p}} \nonumber\\
		&\times\left(\int_{0}^{t} \int_{\mathbb{R}}|h(s, y)|^{\left(1-\frac{2}{p}\right) \frac{2 p}{p-2}} \mathrm{d}s \mathrm{d}y\right)^{\frac{p-2}{2 p}} \nonumber\\
		&\times\left(\int_{0}^{t} \int_{\mathbb{R}}\left(p_{t-s}(x, y) e^{\varphi(s)|y|} e^{-\varphi(t)|x|}\right)^{2} \mathrm{d}s\mathrm{d}y\right)^{\frac{1}{2}} \nonumber\\
		\leq &\sqrt{2}\pi^{-\frac{1}{4}} t^{\frac{1}{4}}e^{ \frac{t\varphi^2(t)}{2}} |h|_{\mathcal{H}}^{\frac{p-2}{p}}\left(\int_{0}^{t}\int_{\mathbb{R}}|\tilde{\sigma}(s, y)|^{p} e^{-p \varphi(s)|y|}|h(s, y)|^{2} \mathrm{d}s\mathrm{d}y \right)^{\frac{1}{p}}.
	\end{align*}
Taking both sides of the above inequality to the power of $p$, we obtain
	\begin{align}\label{casep>2}
			&\sup_{t\leq T, x \in\mathbb{R}}\left\{ \left| \int_{0}^{t}\int_{\mathbb{R}} p_{t-s}(x,y)\tilde{\sigma}(s,y)h(s,y)\mathrm{d}s\mathrm{d}y \right|^p e^{-p\varphi(t)|x|} \right\} \nonumber\\
			\leq& C_{T,\varphi(T),|h|_{\mathcal{H}},p} \int_{0}^{T}\int_{\mathbb{R}}e^{-p\varphi(s)|y|}|\tilde{\sigma}(s,y)|^p|h(s,y)|^2\mathrm{d}s\mathrm{d}y \nonumber\\
			\leq & C_{T,\varphi(T),|h|_{\mathcal{H}},p}\int_{0}^{T} \left(\int_{\mathbb{R}}\left|h(s,y)\right|^2\mathrm{d}y\right)\sup_{y\in\mathbb{R}} \left(\left|\tilde{\sigma}(s,y)\right|e^{-\varphi(s)|y|}\right)^{p} \mathrm{d}s,
	\end{align}
completing the prove of  \eqref{moment1} for the case of $p>2$.

    For $0<p\leq 2$ , let $q=3$, \eqref{casep>2} implies that, for any $\eta>0$,
	\begin{align*}
		&\sup_{t\leq T, x \in\mathbb{R}}\left\{ \left| \int_{0}^{t}\int_{\mathbb{R}} p_{t-s}(x,y)\tilde{\sigma}(s,y)h(s,y)\mathrm{d}s\mathrm{d}y \right|^p e^{-p\varphi(t)|x|} \right\}\\
		\leq &\left\{ C_{T,\varphi(T),|h|_{\mathcal{H}},q} \int_{0}^{T} \int_{\mathbb{R}} |\tilde{\sigma}(s,y)|^q e^{-q\varphi(s)|y|}|h(s,y)|^2\mathrm{d}s\mathrm{d}y \right\}^{\frac{p}{q}}\\
		\leq & C_{T,\varphi(T),|h|_{\mathcal{H}},p,q} \sup_{s\leq T, y\in\mathbb{R}} \left(|\tilde{\sigma}(s,y)|e^{-\varphi(s)|y|}\right)^{\frac{(q-p)p}{q} }
		\left(\int_{0}^{T} \int_{\mathbb{R}}|\tilde{\sigma}(s,y)|^pe^{-p\varphi(s)|y|}|h(s,y)|^2\mathrm{d}s\mathrm{d}y\right)^{\frac{p}{q}}\\
		\leq &\eta\sup_{s\leq T, y\in\mathbb{R}} \left(|\tilde{\sigma}(s,y)|e^{-\varphi(s)|y|}\right)^p+C_{\eta,T,\varphi(T),|h|_{\mathcal{H}},p,q} \int_{0}^{T} \int_{\mathbb{R}}|\tilde{\sigma}(s,y)|^pe^{-p\varphi(s)|y|}|h(s,y)|^2\mathrm{d}s\mathrm{d}y,
	\end{align*}
	here we have used the Young's inequality in the last inequality, completing the prove of  \eqref{moment1} for the case of $0<p\leq 2$.

The proof of this proposition is complete.
\end{proof}

\vskip 0.2cm
Now we are in a position to prove
the main result in this subsection.
\begin{proposition}\label{221105.1303}
	Under the assumptions of Proposition \ref{Propo well posed SE 02},  the equation \eqref{ske0}  has the pathwise uniqueness property.
\end{proposition}
\begin{proof}

Suppose that $Y_1,Y_2 \in C\left([0,T],C_{tem}\right)$ are two solutions of equation (\ref{ske0}),
then
\begin{align}\label{Y1-Y2}
	Y_1(t,x)-Y_2(t,x)=&\int_{0}^{t}\int_{\mathbb{R}} p_{t-s}(x,y) \left[b(Y_1(s,y))-b(Y_2(s,y))\right]\mathrm{d}s\mathrm{d}y \nonumber\\
	&+\int_{0}^{t}\int_{\mathbb{R}} p_{t-s}(x,y)\left[\sigma(Y_1(s,y))-\sigma(Y_2(s,y))\right]h(s,y)\mathrm{d}s\mathrm{d}y.
\end{align}
Recall $\lambda_{T}$ appearing in the proof of Lemma \ref{u_nbdd}; see the above of (\ref{T*}). Now fix $\lambda\in(0,\lambda_T]$, we have $T \leq T^{*}\left(c_{4},\lambda\right)$. In the following we write $\beta\left(c_{4},\lambda\right)$ as $\beta$ for simplicity, where $c_4$ is the constant appearing in {\rm(\textbf{H1(b)})}. Let $0<\delta \leq e^{-1}$. Define
	\begin{align*}
		t^{\delta}:=\inf \left\{t>0: \sup _{x \in \mathbb{R}}\left(|Y_1(t, x)-Y_2(t, x)| e^{-\lambda|x| e^{\beta t}}\right) \geq \delta\right\}\wedge T>0.
	\end{align*}
Set
	\begin{align*}
		g(r)=\sup_{t \leq r \wedge t^{\delta}, x \in\mathbb{R}}\left( |Y_1(t,x)-Y_2(t,x)|e^{-\lambda |x|e^{\beta t}} \right).
	\end{align*}
Then $g(r)< \infty$ and
\begin{align}\label{221105.1945}
	g(r) \leq & \sup_{t \leq r \wedge t^{\delta}, x \in\mathbb{R}} \left\{\int_{0}^{t}\int_{\mathbb{R}} p_{t-s}(x,y)\left|b(Y_1(s,y))-b(Y_2(s,y))\right|\mathrm{d}s\mathrm{d}y\,e^{-\lambda|x|e^{\beta t}}\right\} \nonumber\\
	& + \sup_{t \leq r \wedge t^{\delta}, x \in\mathbb{R}} \left\{\int_{0}^{t}\int_{\mathbb{R}} p_{t-s}(x,y)| \sigma(Y_1(s,y))-\sigma(Y_2(s,y))| |h(s,y)|\mathrm{d}s\mathrm{d}y \,e^{-\lambda|x|e^{\beta t}}\right\} \nonumber\\
	=:&J_1+J_2.
\end{align}
Since $h\in \mathcal{H}$,
\begin{align}\label{gamma}
	\gamma(s):=\int_{\mathbb{R}}\left|h(s,y)\right|^2\mathrm{d}y
\end{align}
is integrable on $[0,T]$.
Applying Proposition \ref{moment} with $p=1$, the term $J_2$ can be estimated by
\begin{align}\label{221105.1946}
	J_2 \leq & \eta\sup_{s \leq r \wedge t^{\delta}, y \in\mathbb{R}}\left(\left|\sigma(Y_1(s,y))-\sigma(Y_2(s,y))\right|e^{-\lambda|y|e^{\beta s}}\right) \nonumber\\
			&+C_{\eta,\lambda,\beta,|h|_{\mathcal{H}},T} \int_{0}^{r \wedge t^{\delta}} \gamma(s) \sup_{y\in\mathbb{R}} \left(\left|\sigma(Y_1(s,y))-\sigma(Y_2(s,y))\right|e^{-\lambda|y|e^{\beta s}}\right) \mathrm{d}s \nonumber\\
			\leq & \eta L_{\sigma} g(s) + C_{\eta,\lambda,\beta,|h|_{\mathcal{H}}, L_{\sigma},T}\int_0^r \gamma(s) g(s)\mathrm{d}s.
\end{align}
We now estimate the term $J_1$.
The assumption ({\textbf{H1(b)}}) gives that
\begin{align}\label{uniq1}
J_1 \leq & c_{3}\sup_{t \leq r \wedge  t^\delta, x \in \mathbb{R}}\left\{\int_{0}^{t} \int_{\mathbb{R}} p_{t-s}(x, y) |Y_1(s,y)-Y_2(s,y)|\log_{+} \frac{1}{|Y_1(s,y)-Y_2(s,y)|} \mathrm{d} s \mathrm{d}y \,e^{-\lambda|x| e^{\beta t}}\right\} \nonumber\\
	&+c_{4}\sup _{t \leq r \wedge  t^\delta, x \in \mathbb{R}}\left\{\int_{0}^{t} \int_{\mathbb{R}} p_{t-s}(x, y)  \log_+ \left(|Y_1(s,y)|\vee|Y_2(s,y)|\right)|Y_1(s,y)-Y_2(s,y)| \mathrm{d}s\mathrm{d}y\,e^{-\lambda|x| e^{\beta t}}\right\} \nonumber\\
	&+c_{5}\sup _{t \leq r \wedge  t^\delta, x \in \mathbb{R}}\left\{\int_{0}^{t} \int_{\mathbb{R}} p_{t-s}(x, y) |Y_1(s,y)-Y_2(s,y)|\mathrm{d}s \mathrm{d} y \,e^{-\lambda|x| e^{\beta t}}\right\} \nonumber\\
	=:& J_{11}+J_{12}+J_{13}.
\end{align}
Using the fact that the function $x \mapsto x \log_+ \frac{1}{x}$ is increasing on $\left(0, e^{-1}\right)$ and (\ref{4.37}), we get
\begin{align}\label{uniq11}
	J_{11} \leq& c_{3} \sup_{t \leq r \wedge t^\delta, x \in \mathbb{R}}\left\{\int_{0}^{t} \int_{\mathbb{R}} p_{t-s}(x, y) e^{\lambda|y| e^{\beta s}}|Y_1(s,y)-Y_2(s,y)| \,e^{-\lambda|y| e^{\beta s}} \right.\nonumber\\
	&\qquad \left.\times\log_{+} \frac{1}{|Y_1(s,y)-Y_2(s,y)|e^{-\lambda|y| e^{\beta s}}} \mathrm{d} s \mathrm{d} y e^{-\lambda|x| e^{\beta t}}\right\} \nonumber\\
	\leq &c_{3} \sup_{t \leq r \wedge t^\delta, x \in \mathbb{R}}\left\{\int_{0}^{t} \sup_{y \in \mathbb{R}}\left(|Y_1(s,y)-Y_2(s,y)|\,e^{-\lambda|y| e^{\beta s}}\log_{+} \frac{1}{|Y_1(s,y)-Y_2(s,y)|e^{-\lambda|y| e^{\beta s}} } \right)\right.\nonumber\\
	&\qquad\left.\times\int_{\mathbb{R}} p_{t-s}(x, y) e^{\lambda|y| e^{\beta s}} \mathrm{d}y \mathrm{d}s \,e^{-\lambda|x| e^{\beta t}}\right\} \nonumber\\
	\leq &2 c_{3} e^{\frac{\lambda^{2}}{4 \beta} e^{2 \beta r-1}} \int_{0}^{r} \sup_{\rho \leq s \wedge t^\delta, y \in \mathbb{R}}\left(|Y_1(\rho, y)-Y_2(\rho, y)| e^{-\lambda|y| e^{\beta \rho}}\right) \nonumber\\
	&\qquad \times\log_{+} \frac{1}{\sup_{\rho \leq s \wedge t^\delta, y \in \mathbb{R}}\left(|Y_1(\rho, y)-Y_2(\rho, y)| e^{-\lambda|y| e^{\beta \rho}}\right)} \mathrm{d} s\nonumber\\
	\leq& 2 c_{3} e^{\frac{\lambda^{2}}{4 \beta} e^{2 \beta r-1}} \int_{0}^{r} g(s) \log_{+} \frac{1}{g(s)} \mathrm{d} s.
\end{align}
Similarly, we can prove that
\begin{align}\label{uniq13}
	J_{13} \leq 2 c_{5} e^{\frac{\lambda^{2}}{4 \beta} e^{2 \beta r-1}} \int_{0}^{r} g(s) \mathrm{d} s .
\end{align}
The inequality $\log_+(ab)\leq\log_+a+\log_+b$ leads to
\begin{align}\label{uniq12}
	J_{12} \leq& c_{4} \sup _{t\leq r \wedge t^\delta, x \in \mathbb{R}}\left\{\int_{0}^{t} \int_{\mathbb {R} } p_{t-s}(x,y)e^{\lambda|y| e^{\beta s}}  \left[\log_{+}\left(\left(|Y_1(s,y)| e^{-\lambda|y| e^{\beta s}}\right) \vee\left(|Y_2(s,y)| e^{-\lambda|y| e^{\beta s}}\right)\right)\right.\right. \nonumber\\
	&\left.\left.+\lambda|y| e^{\beta s}\right]|Y_1(s,y)-Y_2(s,y)| e^{-\lambda|y| e^{\beta s}}\mathrm{d}s\mathrm{d}y  \,e^{-\lambda|x| e^{\beta t}}\right\} \nonumber \\
	\leq & c_{4}\sup _{t \leq r \wedge t^\delta, x \in \mathbb{R}}\left\{\int_{0}^{t} \sup_{y \in \mathbb{R}}\left(|Y_1(s,y)-Y_2(s,y)| e^{-\lambda|y| e^{\beta s}}\right) \int_{\mathbb{R}} p_{t-s}(x, y) e^{\lambda|y| e^{\beta s}} \lambda|y| e^{\beta s}  \mathrm{d}y\mathrm{d}s \,e^{-\lambda|x| e^{\beta t}}\right\}\nonumber \\
	&+c_{4} (\log_{+} M) \sup_{t \leq r \wedge t^\delta, x \in \mathbb{R}}\left\{\int_{0}^{t} \sup_{y\in\mathbb{R}} \left(|Y_1(s,y)-Y_2(s,y)| e^{-\lambda|y| e^{\beta s}}\right)\int_{\mathbb{R}} p_{t-s}(x, y) e^{\lambda|y| e^{\beta s}} \mathrm{d}y\mathrm{d}s \,e^{-\lambda|x| e^{\beta t}}\right\} \nonumber\\
	\leq & \frac{c_{4}}{\beta} e^{\frac{\lambda ^{2}}{4 \beta} e^{2 \beta T-1}}g(r)
	+C_{c_4,\lambda,M,T}\int_{0}^{r} g(s) \mathrm{d}s,
\end{align}
where
\begin{align*}
	M:=\max\Big\{\sup_{s\leq T, y\in\mathbb{R}}\left(|Y_1(s,y)| e^{-\lambda|y| e^{\beta s}}\right), \sup_{s\leq T, y\in\mathbb{R}}\left(|Y_2(s,y)| e^{-\lambda|y| e^{\beta s}}\right)\Big\} .
\end{align*}
Recall that we take $\lambda>0$ sufficiently small such that $T \leq T^{*}\left(c_{4},\lambda\right)$ and $\beta=\beta(c_{4},\lambda)$. Using the fact
\begin{align*}
	\frac{c_{4}}{\beta} e^{\frac{\lambda ^{2}}{4 \beta} e^{2 \beta T-1}} \leq \frac{1}{2} \Longleftrightarrow T \leq T^{*}\left( c_{4},\lambda \right)=\frac{1}{2 \beta}\left[1+\log \left(\frac{4 \beta}{\lambda^{2}} \log \frac{\beta}{2 c_{4}}\right)\right],
\end{align*}
we obtain
\begin{align}\label{221105.1947}
	J_{12}\leq \frac{1}{2} g(r)+C_{c_4,\lambda,M,T} \int_{0}^{r} g(s)\mathrm{d}s.
\end{align}
Combining (\ref{221105.1945}), (\ref{221105.1946})-(\ref{uniq13}) and (\ref{221105.1947}) together yields
\begin{align}\label{uniq}
		g(r)\le &\frac{1}{2} g(r)+C_{c_4,c_5,\lambda, M,T} \int_{0}^{r} g(s)\mathrm{d}s+C_{c_3,c_4,\lambda,T} \int_{0}^{r} g(s)\log_+\frac{1}{g(s)}\mathrm{d}s \nonumber\\
		&+\eta L_\sigma g(r)+C_{\eta,c_4,\lambda,|h|_{\mathcal{H}},L_{\sigma},T}\int_{0}^{r} \gamma(s)g(s)\mathrm{d}s.
\end{align}
Taking $\eta< \frac{1}{2L_{\sigma}}$ and by Lemma \ref{Gronwall2}, we get $g(r)\equiv 0$ for any $r\in[0,t^\delta]$. It follows by the definition of $t^\delta$   that $t^\delta= T$. And then $Y_1(t,x)=Y_2(t,x),\ \forall \ (t,x)\in[0,T]\times\mathbb{R}$, completing the proof of this proposition.
\end{proof}

\section{Verifications of Claims (\textbf{C1}) and (\textbf{C2}): Case 1}\label{221107.1429}
\setcounter{equation}{0}

In this section, we assume that

\begin{center}
{\bf Assumption Case 1}:\ \ \ \ \ \ \ \ \  $u_0\in C_{tem}$, {\rm(\textbf{H1(b)})} and {\rm(\textbf{H1(c)})} hold.
\end{center}
Under this assumption, we will verify Claims (\textbf{C1}) and (\textbf{C2}) in the proof of Theorem \ref{221107.1426}; see Propositions \ref{C1} and \ref{C2}, respectively.

\subsection{Verification of (\textbf{C1})}

Recall the definition of $\mathcal{G}^0$ introduced in the proof of Theorem \ref{221107.1426}.
In this subsection, we verify Claim (\textbf{C1}). That is
\begin{proposition}\label{C1}
Under {\bf Assumption Case 1}, for any $N>0$, $\{h_n\}_{n\in\mathbb{N}}\subset \mathcal{H}_N$ and $h\in\mathcal{H}_N$ with $h_n\to h$ weakly in $\mathcal{H}$ as $n\to\infty$, then $\lim_{n\rightarrow\infty}\mathcal{G}^0(\text{Int}(h_n))=\mathcal{G}^0(\text{Int}(h))$ in $C([0,T],C_{tem})$.
%
%
\end{proposition}

\vskip 0.2cm

%

In this subsection, we denote $Y^h:=\mathcal{G}^0(\text{Int}(h))$ and $Y^{h_n}:=\mathcal{G}^0(\text{Int}(h_n))$ for simplicity.
%
%
%
Before the proof Proposition \ref{C1}, we state Lemma \ref{u^nbdd} and Lemma \ref{U^nV^n} below, whose proofs are similar to Lemma \ref{u_nbdd} and Lemma \ref{UnVn} respectively and hence  omitted here.
\begin{lemma}\label{u^nbdd}
For any $T, \lambda>0$,
	\begin{align}\label{221106.2100}
	\sup_{n\in\mathbb{N}}\sup_{t\leq T,x\in\mathbb{R}} \left(|Y^{h_n}(t,x)|e^{-\lambda|x|}\right)\leq C_{c_1,c_2,\lambda, u_0,K_\sigma,N,T}.
	\end{align}
\end{lemma}

\vskip 0.5cm
Set
\begin{align}
	F_n(t,x):=\int_{0}^{t}\int_{\mathbb{R}}p_{t-r}(x,z)\sigma(Y^h(r,z))\left(h_n(r,z)-h(r,z)\right)\mathrm{d}r\mathrm{d}z.
\end{align}
\begin{lemma}\label{U^nV^n}
	There exists a constant $C_{K_{\sigma},N}$ such that, for any $s,t \in[0,T]$ and $x, y \in \mathbb{R}$,
	\begin{align*}
\sup_{n\in\mathbb{N}}|F_n(t,x) - F_n(s,y)|\leq & C_{K_{\sigma}, N}\left(|t-s|^{\frac{1}{4}}+|x-y|^{\frac{1}{2}}\right).
	\end{align*}
\end{lemma}


\vskip 0.5cm

\begin{proof}[Proof of Proposition \ref{C1}]
	Recall $T^*$ and $c_1,c_4$ introduced in (\ref{T}) and \textbf{Hypothesis 1}, respectively. For any fixed $T>0$, (\ref{eq T}) implies that there exists a positive constant $\lambda_{T}$ such that $T \leq T^{*}\left( c_{1}\vee c_{4},\lambda\right)$ for all $\lambda \in(0, \lambda_{T}]$.
Write $\beta\left(c_1\vee c_{4},\lambda\right)$ and  as $\beta$ for simplicity. To prove
Proposition \ref{C1}, it suffices to show that, for any $\lambda\in(0,\lambda_T]$,
	\begin{align*}
		\lim_{n\to\infty}\sup_{t\leq T,x\in\mathbb{R}} \left(|Y^{h_n}(t,x)-Y^h(t,x)|e^{-\lambda|x|e^{\beta t}}\right)=0.
	\end{align*}
	 Fix $\lambda\in(0,\lambda_T]$. Let $0<\delta \leq e^{-1}$,
	\begin{align*}
		t^{\delta}_n :=\inf \left\{t>0: \sup _{x \in \mathbb{R}}\left(|Y^{h_n}(t, x)-Y^h(t, x)| e^{-\lambda|x| e^{\beta t}}\right) \geq \delta\right\}\wedge T,
	\end{align*}
	and
	\begin{align*}
	g_n(r):=\sup_{t\leq r\wedge t^{\delta}_n,x\in\mathbb{R}} \left(|Y^{h_n}(t,x)-Y^h(t,x)|e^{-\lambda|x|e^{\beta t}}\right).
	\end{align*}
	It follows by \eqref{ske0} that
	\begin{align}\label{221105.2301}
g_n(r)\leq & \sup_{t \leq r\wedge t^{\delta}_n, x \in \mathbb{R}}\left\{\left|\int_{0}^{t} \int_{\mathbb{R}} p_{t-s}(x,y)\left[b(Y^{h_n}(s,y))-b(Y^h(s,y))\right]\mathrm{d}s\mathrm{d}y\right|e^{-\lambda |x|e^{\beta t}}\right\} \nonumber\\
		&+\sup_{t \leq r\wedge t^{\delta}_n, x \in \mathbb{R}}\left\{\left|\int_{0}^{t}\int_{\mathbb{R}}p_{t-s}(x,y)\left[\sigma\left(Y^{h_n}(s,y)\right)h_n(s,y)-\sigma(Y^h(s,y))h(s,y)\right]\mathrm{d}s\mathrm{d}y\right|e^{-\lambda |x|e^{\beta t}}\right\}\nonumber\\
		=: & R + Q.
	\end{align}
The term $Q$ is bounded by
	\begin{align}\label{221105,2303}
		Q \leq &\sup_{t \leq r\wedge t^{\delta}_n,x\in\mathbb{R}}\left\{\left| \int_{0}^{t}\int_{\mathbb{R}}p_{t-s}(x,y)\sigma(Y^h(s,y))\left[ h_n(s,y)-h(s,y) \right]\mathrm{d}s\mathrm{d}y\right| \,e^{-\lambda|x|e^{\beta t}}\right\} \nonumber\\
		&+ \sup_{t\leq r\wedge t^{\delta}_n, x\in\mathbb{R}}\left\{\left| \int_{0}^{t}\int_{\mathbb{R}}p_{t-s}(x,y)\left[ \sigma(Y^{h_n}(s,y))- \sigma(Y^h(s,y)) \right]h_n(s,y)\mathrm{d}s \mathrm{d}y\right|\,e^{-\lambda|x|e^{\beta t}}\right\} \nonumber\\
		=:& Q_1+Q_2.
	\end{align}
We estimate the term $Q_1$ as follows,
\begin{align}\label{221105.2216}
	Q_1\leq & \sup_{t \leq T, |x|>M} \left\{\left|\int_{0}^{t}\int_{\mathbb{R}}p_{t-s}(x,y)\sigma(Y^h(s,y))\left(h_n(s,y)-h(s,y)\right)\mathrm{d}s\mathrm{d}y\right|e^{-\lambda|x|e^{\beta t} }\right\} \nonumber\\
	&+\sup_{t \leq T, |x|\leq M}\left\{ \left|\int_{0}^{t}\int_{\mathbb{R}}p_{t-s}(x,y)\sigma(Y^h(s,y))\left(h_n(s,y)-h(s,y)\right)\mathrm{d}s\mathrm{d}y\right|e^{-\lambda|x|e^{\beta t} }\right\} \nonumber\\
	\leq &C_{K_{\sigma},N,T} e^{-\lambda M}+\sup_{t \leq T, |x|\leq M} \left|\int_{0}^{t}\int_{\mathbb{R}}p_{t-s}(x,y)\sigma(Y^h(s,y))\left(h_n(s,y)-h(s,y)\right)\mathrm{d}s\mathrm{d}y\right|,
\end{align}
where the boundedness of $\sigma$ and the similar argument in the proof of \eqref{fn3} have been used. According to Lemma \ref{U^nV^n}, for any $M>0$,
\begin{align*}
	F_n(t,x):= \int_{0}^{t}\int_{\mathbb{R}}p_{t-s}(x,y)\sigma(Y^h(s,y))\left(h_n(s,y)-h(s,y)\right)\mathrm{d}s\mathrm{d}y
\end{align*}
is equicontinuous on $[0,T]\times[-M,M]$. That is, for any $\varepsilon>0$, there exists $\rho>0$ such that for any $t,t'\in[0,T], x,x'\in [-M,M]$ with $|t-t'|<\rho, |x-x'|<\rho$,
\begin{align*}
	\sup_{n\in\mathbb{N}}\left|F_n(t,x)-F_n(t',x')\right|<\varepsilon,
\end{align*}
which implies that
\begin{align}\label{221105.2217}
	\sup_{t \leq T, |x|\leq M} | F_n(t,x) | \leq \!\!\!\!\!\!\!\!\sum_{\substack{i=1,2,\cdots,\lfloor T\rho\rfloor\\j=-\lfloor M\rho \rfloor, -\lfloor M\rho \rfloor +1\cdots ,\lfloor M\rho \rfloor} } \!\!\!\!\left|\int_{0}^{i\rho}\int_{\mathbb{R}}p_{i\rho-s}(j\rho,y)\sigma(Y^h(s,y))(h_n(s,y)-h(s,y))\mathrm{d}s\mathrm{d}y\right|+\varepsilon.
\end{align}
Combining (\ref{221105.2216}) and (\ref{221105.2217}) together yields
\begin{align*}
	Q_1\leq &C_{K_\sigma,N,T} e^{-\lambda M}+\varepsilon \nonumber\\
	&+\sum_{\substack{i=1,2,\cdots,\lfloor T\rho\rfloor\\j= -\lfloor M\rho \rfloor, -\lfloor M\rho \rfloor +1\cdots ,\lfloor M\rho \rfloor} } \left|\int_{0}^{i\rho}\int_{\mathbb{R}}p_{i\rho-s}(j\rho,y)\sigma(Y^h(s,y))(h_n(s,y)-h(s,y))\mathrm{d}s\mathrm{d}y\right|.
\end{align*}
Letting $n\to\infty$ and then $M\to +\infty$, we have
\begin{align*}
	\limsup_{n\to \infty}Q_1\leq \varepsilon,
\end{align*}
here $\lim_{n\to\infty}h_n= h$ in $\mathcal{H}_N$ has been used. Then the arbitrariness of $\varepsilon$ implies that
\begin{align}\label{C11}
	\lim_{n\to \infty}Q_1=0.
\end{align}
For the term $Q_2$, by Proposition \ref{moment} with $p=1$ and the Lipschitz continuity of $\sigma$ (see (\textbf{H1(c)})), for any $\eta>0$,
\begin{align}\label{221105.2304}
	Q_2\leq & \eta L_\sigma g_n(r)+C_{\eta,c_4,\lambda,L_\sigma,N,T}\int_{0}^{r} \gamma_n(s)g_n(s)\mathrm{d}s,
\end{align}
where
\begin{align}\label{221105.2308}
	\gamma_n(s): = \int_{\mathbb{R}}\left|h_n(s,y)\right|^2\mathrm{d}y
\end{align}
is integrable on $[0,T]$.
The term $R$ can be estimated similarly as \eqref{uniq1}-\eqref{221105.1947}:
\begin{align}\label{221105.2305}
	R\leq & c_3 \sup_{t \leq t^{\delta}_n, x\in\mathbb{R}} \left\{\int_{0}^{t}\int_{\mathbb{R}}p_{t-s}(x,y)|Y^{h_n}(s,y)-Y^h(s,y)|\log_+\frac{1}{|Y^{h_n}(s,y)-Y^h(s,y)|}\mathrm{d}s\mathrm{d}y e^{-\lambda |x|e^{\beta t}}\right\} \nonumber\\
	&+c_4 \sup_{t \leq t^{\delta}_n, x\in\mathbb{R}} \left\{\int_{0}^{t}\int_{\mathbb{R}}p_{t-s}(x,y)\log_+\left(|Y^{h_n}(s,y)|\vee|Y^h(s,y)|\right)|Y^{h_n}(s,y)-Y^h(s,y)|\mathrm{d}s\mathrm{d}y e^{-\lambda |x|e^{\beta t}}\right\} \nonumber\\
	&+c_5\sup_{t \leq t^{\delta}_n, x\in\mathbb{R}} \left\{\int_{0}^{t}\int_{\mathbb{R}}p_{t-s}(x,y)|Y^{h_n}(s,y)-Y^h(s,y)|\mathrm{d}s\mathrm{d}y e^{-\lambda |x|e^{\beta t}}\right\} \nonumber\\
	\leq & 2c_3 e^{\frac{\lambda^2}{4\beta }e^{2\beta r-1}} \int_{0}^{r} g_n(s)\log_+\frac{1}{g_n(s)}\mathrm{d}s+2c_5e^{\frac{\lambda^2}{4\beta}e^{2\beta r-1}} \int_{0}^{r} g_n(s)\mathrm{d}s \nonumber\\
	&+\frac{1}{2}g_n(r)+C_{c_1,c_2,c_{4},\lambda,u_0,K_\sigma,N,T}\int_{0}^{r} g_n(s)\mathrm{d}s ,
\end{align}
here we have used Lemma \ref{u^nbdd}, i.e., $\{Y^{h_n}\}$ is uniformly bounded. Combining (\ref{221105.2301}), (\ref{221105,2303}), (\ref{221105.2304}), and (\ref{221105.2305}) together gives
\begin{align}
	g_n(r)\leq& Q_1+\eta L_\sigma g_n(r)+\frac{1}{2}g_n(r)+C_{\eta,c_4,\lambda,L_{\sigma},N,T}\int_{0}^{r}\gamma_n(s)g_n(s)\mathrm{d}s\nonumber\\
	&+C_{c_1,c_2,c_4,c_5,\lambda,u_0,K_\sigma,N,T}\int_{0}^{r}g_n(s)\mathrm{d}s+C_{c_3,c_4,\lambda,T}\int_{0}^{r} g_n(s)\log_+\frac{1}{g_n(s)}\mathrm{d}s.
\end{align}
In view of (\ref{221105.2308}) and $h_n\in \mathcal{H}_N$, we have $\int_0^T \gamma_n(s)\mathrm{d}s \leq N$. Thus, taking $\eta < \frac{1}{2L_{\sigma}}$ and by Lemma \ref{Gronwall2} and \eqref{C11} we obtain that $g_n(T)\to 0$ as $n\to\infty$.
This implies that $t^{\delta}_n\geq T$ for large enough $n$, otherwise it would contradict the definition of $t^{\delta}_n$. Therefore, we deduce that
\begin{align*}
\sup_{t\leq T,x\in\mathbb{R}} \left(|Y^{h_n}(t,x)-Y^h(t,x)|e^{-\lambda|x|e^{\beta t}}\right)\to 0 \quad \text{ as }n\to\infty,
\end{align*}
and the proof of Proposition \ref{C1} is complete.
\end{proof}

\vskip 0.5cm

\subsection{Verification of (\textbf{C2})}\label{221203.2252}

The claim (\textbf{C2}) is verified in the following proposition.



\begin{proposition}\label{C2}
Under {\bf Assumption Case 1}, for any $N>0$, $\{h_\epsilon,\epsilon>0\}\subset \mathcal{X}_N$ and $\delta>0$,
\begin{align*}
	\lim_{\epsilon\to 0} \mathbb{P} \left(   \tilde{d}\left(X^{\epsilon,h_\epsilon},Y^{h_\epsilon}\right)>\delta  \right)=0,
\end{align*}
where $Y^{h_\epsilon}=\mathcal{G}^{0}(\text{Int}(h_{\epsilon}))$ is the unique solution to (\ref{ske0}) with $h$ replaced by $h_{\epsilon}$, and $X^{\epsilon,h_\epsilon}$ is the unique solution to (\ref{eqnhe}).

\end{proposition}

Before proving Proposition \ref{C2}, we first give  certain uniform estimates on $X^{\epsilon,h_\epsilon}$ and $Y^{h_\epsilon}$. Recall $\beta$ introduced in (\ref{eq beta}).
 In the following, for $\lambda>0$, we write $\beta(c_1\vee c_4,\lambda)$ as $\beta$, where $c_1$ and $c_4$ are the constants appearing in \textbf{Hypothesis 1}.

\begin{lemma}\label{U}
For any $T, \lambda>0$, there exists a constant $C_{c_1,c_2,\lambda,\beta,u_0,K_\sigma,N,T}$ such that
	\begin{align}
\label{221106.2057}		\sup_{0<\epsilon<1}\mathbb{E}\sup_{t\leq T,x\in\mathbb{R}} \left(\left|X^{\epsilon,h_\epsilon}(t,x)\right| e^{-\lambda|x|e^{\beta t}}\right) \leq &\ C_{c_1,c_2,\lambda,\beta,u_0,K_\sigma,N,T} , \\
\label{221106.2049}		\sup_{0<\epsilon<1}\mathbb{E}\sup_{t\leq T,x\in\mathbb{R}} \left(\left|Y^{h_\epsilon}(t,x)\right| e^{-\lambda|x|e^{\beta t}}\right) \leq &\  C_{c_1,c_2,\lambda,\beta,u_0,K_\sigma,N,T}.
	\end{align}
\end{lemma}

\begin{proof}
It is not difficult that $\{h_\epsilon,\epsilon>0\}\subset \mathcal{X}_N$ and (\ref{221106.2100}) imply (\ref{221106.2049}). In the sequel, we only prove (\ref{221106.2057}).


Recall $T^*$ is introduced in (\ref{T}). For any fixed $T>0$, (\ref{eq T}) implies that there exists a positive constant $\lambda_{T}$ such that $T \leq T^{*}\left( c_{1}\vee c_4,\lambda\right)$ for all $\lambda \in(0, \lambda_{T}]$.  We will prove (\ref{221106.2057}) for any $\lambda\in(0,\lambda_T]$, which implies (\ref{221106.2057}) for arbitrary $\lambda>0$.

Now fix $\lambda\in(0,\lambda_T]$. Set
	\begin{align*}
		U^{\epsilon}(r):=\sup_{t\leq r,x\in\mathbb{R}} \left(\left|X^{\epsilon,h_\epsilon}(t,x)\right| e^{-\lambda|x|e^{\beta t}}\right), \quad r\leq T.
	\end{align*}
It follows from (\ref{eqnhe}) that
	\begin{align}\label{221106.1924}
	U^{\epsilon}(r)\leq& \sup_{t\leq r,x\in\mathbb{R}} \left( \left|\int_{\mathbb{R}} p_t(x,y)u_0(y)\mathrm{d}y \right| e^{-\lambda|x|e^{\beta t}}\right) \nonumber\\
	&+\sup_{t\leq r,x\in\mathbb{R}}\left( \left| \int_{0}^{t}\int_{\mathbb{R}}p_{t-s}(x,y)b(X^{\epsilon,h_\epsilon}(s,y))\mathrm{d}s\mathrm{d}y \right| e^{-\lambda|x|e^{\beta t}}\right) \nonumber\\
	&+\sup_{t\leq r,x\in\mathbb{R}} \left( \left| \int_{0}^{t}\int_{\mathbb{R}} p_{t-s}(x,y)\sigma(X^{\epsilon,h_\epsilon}(s,y))h_{\epsilon}(s,y)\mathrm{d}s\mathrm{d}y \right| e^{-\lambda|x|e^{\beta t}}\right) \nonumber\\
	&+\sqrt{\epsilon}\sup_{t\leq r,x\in\mathbb{R}} \left(\left| \int_{0}^{t}\int_{\mathbb{R}} p_{t-s}(x,y)\sigma(X^{\epsilon,h_\epsilon}(s,y))W(\mathrm{d}s,\mathrm{d}y)\mathrm{d}s\mathrm{d}y \right| e^{-\lambda|x|e^{\beta t}}\right).
\end{align}
By Assumption (\textbf{H1(b)}) and using the similar arguments as proving (\ref{fn2})-(\ref{fn23}), we have
\begin{align}
	&\sup_{t\leq r,x\in\mathbb{R}}\left( \left| \int_{0}^{t}\int_{\mathbb{R}}p_{t-s}(x,y)b(X^{\epsilon,h_\epsilon}(s,y))\mathrm{d}s\mathrm{d}y \right|  e^{-\lambda|x|e^{\beta t}}\right) \nonumber\\
	\leq &\sup_{t\leq r,x\in\mathbb{R}}\left\{\int_{0}^{t}\int_{\mathbb{R}}p_{t-s}(x,y)\left(c_1|X^{\epsilon,h_\epsilon}(s,y)|\log_+|X^{\epsilon,h_\epsilon}(s,y)|+c_2\right)\mathrm{d}s\mathrm{d}ye^{-\lambda|x|e^{\beta t}}\right\} \nonumber \\
	\leq&c_2T + c_1\sup_{t\leq r,x\in\mathbb{R}} \bigg\{ \int_0^t \sup_{y\in\mathbb{R}} \left[ |X^{\epsilon,h_\epsilon}(s,y)|e^{-\lambda |y|e^{\beta s}}\log_+\left(|X^{\epsilon,h_\epsilon}(s,y)|e^{-\lambda |y|e^{\beta s}}\right)\right] \nonumber\\
	& ~~~~~~~~~~~~~~~~~~~~~~ \times \int_{\mathbb{R}}p_{t-s}(x,y)e^{\lambda|y|e^{\beta s}}\mathrm{d}y \mathrm{d}s \cdot e^{-\lambda|x|e^{\beta t}} \bigg\} \nonumber\\
	& + c_1\sup_{t\leq r,x\in\mathbb{R}} \left\{\int_0^t \sup_{y\in\mathbb{R}} \left( |X^{\epsilon,h_\epsilon}(s,y)|e^{-\lambda |y|e^{\beta s}}\right)  \int_{\mathbb{R}}p_{t-s}(x,y)e^{\lambda|y|e^{\beta s}}\lambda|y|e^{\beta s}\mathrm{d}y \mathrm{d}s \cdot e^{-\lambda|x|e^{\beta t}} \right\} \nonumber\\
	\leq & c_2T+2c_1e^{\frac{\lambda^2}{4\beta}e^{2\beta T-1}} 	\int_{0}^{r} U^{\epsilon}(s)\log_+U^{\epsilon}(s)\mathrm{d}s+\frac{c_1}{\beta} e^{\frac{\lambda^2}{4\beta}e^{2\beta T-1}} U^{\epsilon}(r)+C_{c_1,\lambda,\beta,T}\int_{0}^{r} U^{\epsilon}(s)\mathrm{d}s \nonumber\\
	\leq & c_2T+2c_1e^{\frac{\lambda^2}{4\beta}e^{2\beta T-1}} \int_{0}^{r} U^{\epsilon}(s)\log_+U^{\epsilon}(s)\mathrm{d}s + \frac{c_1}{2(c_1\vee c_4)} U^{\epsilon}(r) + C_{c_1,\lambda,\beta,T}\int_{0}^{r} U^{\epsilon}(s)\mathrm{d}s ,
\end{align}
here the last inequality has used the facts that for any $\lambda\in(0,\lambda_T]$, $T\leq T^*(c_1\vee c_4,\lambda)$,
\begin{align*}
	\frac{c_1\vee c_4}{\beta}e^{\frac{\lambda^2}{4\beta}e^{2\beta T-1}}\leq \frac{1}{2} \iff T\leq T^*(c_1\vee c_4,\lambda).
\end{align*}
Applying H\"older's inequality, the boundedness of $\sigma$, (\ref{3.21}), and $h_\epsilon\in\mathcal{H}_N$, $\mathbb{P}$-a.e. yield that
\begin{align}
	&\sup_{t\leq T,x\in\mathbb{R}} \left( \left| \int_{0}^{t}\int_{\mathbb{R}} p_{t-s}(x,y)\sigma(X^{\epsilon,h_\epsilon}(s,y))h_{\epsilon}(s,y)\mathrm{d}s\mathrm{d}y \right| e^{-\lambda|x|e^{\beta t}}\right) \nonumber\\
	\leq&K_\sigma N \sup_{t\leq T,x\in\mathbb{R}}\left(\int_{0}^{t} \int_{\mathbb{R}}p^2_{t-s}(x,y)\mathrm{d}s\mathrm{d}y\right)^{\frac{1}{2}} \nonumber\\
	\leq& \sqrt{2}\pi^{-\frac{1}{4}}K_\sigma N T^{\frac{1}{4}},\ \mathbb{P}\text{-a.e.}
\end{align}
The estimate of the first term on the right hand side of (\ref{221106.1924}) is the same as (\ref{fn1}). Therefore, for $r\leq T$,
\begin{align}
	U^{\epsilon}(r) & \leq \frac{1}{2}U^{\epsilon}(r)+c_2T+\sqrt{2}\pi^{-\frac{1}{4}}K_\sigma NT^{\frac{1}{4}}+2e^{\frac{\lambda^2 T}{2}}|u_0|_{(-\lambda)} \nonumber\\
	&+\sqrt{\epsilon}\sup_{t\leq T,x\in\mathbb{R}} \left( \left| \int_{0}^{t}\int_{\mathbb{R}} p_{t-s}(x,y)\sigma(X^{\epsilon,h_\epsilon}(s,y))W(\mathrm{d}s,\mathrm{d}y) \right| e^{-\lambda|x|e^{\beta t}}\right) \nonumber\\
	&+2c_1e^{\frac{\lambda^2}{4\beta}e^{2\beta T-1}} 		\int_{0}^{r} U^{\epsilon}(s)\log_+U^{\epsilon}(s)\mathrm{d}s+C_{c_1,\lambda,\beta,T}\int_{0}^{r} U^{\epsilon}(s)\mathrm{d}s,\ \mathbb{P}\text{-a.e.}
\end{align}
It follows from Lemma \ref{Gronwall1} that
\begin{align}
	U^{\epsilon}(T)\leq &\bigg\{ C_{c_1,c_2,\lambda,\beta,u_0,K_{\sigma},N,T} \nonumber\\
	& +\sqrt{\epsilon}\sup_{t\leq T,x\in\mathbb{R}} \left(\left| \int_{0}^{t}\int_{\mathbb{R}} p_{t-s}(x,y)\sigma(X^{\epsilon,h_\epsilon}(s,y))W(\mathrm{d}s,\mathrm{d}y) \right| e^{-\lambda|x|e^{\beta t}}\right)\bigg\}^{C_{c_1,\lambda,\beta,T}},\ \mathbb{P}\text{-a.e.}
\end{align}
Here the constant $C_{c_1,\lambda,\beta,T}\ge1$. By the inequality $(a+b)^p\le2^{p-1}(a^p+b^p),\ p\ge 1$, we get
\begin{align}\label{U1}
	\mathbb{E}U^{\epsilon}(T)\leq& C_{c_1,c_2,\lambda,\beta,u_0,K_{\sigma},N,T} \nonumber\\
	&\times\left\{ 1+\epsilon^{\frac{C_{c_1,\lambda,\beta,T}}{2}}\mathbb{E}\left[\sup_{t\leq T,x\in\mathbb{R}} \left( \left| \int_{0}^{t}\int_{\mathbb{R}} p_{t-s}(x,y)\sigma(X^{\epsilon,h_\epsilon}(s,y))W(\mathrm{d}s,\mathrm{d}y) \right| e^{-\lambda|x|e^{\beta t}}\right)^{C_{c_1,\lambda,\beta,T}}\right]\right\}.
\end{align}
According to Lemma 4.1 and Proposition 4.2 in \cite{SZ}, for any $p>0$,
\begin{align}\label{U2}
\mathbb{E}\sup_{t\leq T, x\in\mathbb{R}}\left( \left| \int_{\mathbb{R}} p_{t-s}(x,y)\sigma(X^{\epsilon,h_\epsilon}(s,y))W(\mathrm{d}s,\mathrm{d}y) \right| e^{-\lambda|x|e^{\beta t}}\right)^p\leq C_{\lambda,\beta,K_\sigma,T,p}.
\end{align}
Combining (\ref{U1}) with (\ref{U2}) together, we obtain that for $0<\epsilon<1$,
\begin{align*}
	\mathbb{E}U^{\epsilon}(T)\leq C_{c_1,c_2,\lambda,\beta,u_0,K_\sigma,N,T},
\end{align*}
and the constant on the right hand side is independent of $\epsilon$. The proof of (\ref{221106.2057}) is complete.

The proof of Lemma \ref{U} is complete.
\end{proof}

\vskip 0.5cm

\begin{proof}[Proof of Proposition \ref{C2}]
(\ref{eq T}) implies that there exists a positive constant $\lambda_{T}$ such that $T \leq T^{*}\left( c_{1},\lambda\right)$ for all $\lambda \in(0, \lambda_{T}]$.
Note that for $\varphi,\varphi_n\in C_{tem},n\ge1$,
\begin{align}
	\tilde{d}(\varphi_n,\varphi)&=\sup_{t\leq T} \sum_{k=1}^{\infty} \min\left\{ 1,\sup_{x\in\mathbb{R}}\left(|\varphi_n(t,x)-\varphi(t,x)|e^{-\frac{1}{k}|x|} \right)\right\} \nonumber\\
	&\leq \sum_{k=1}^{K} \frac{1}{2^k} \sup_{t\leq T, x\in\mathbb{R}}\left(|\varphi_n(t,x)-\varphi(t,x)|e^{-\frac{1}{k}|x|}\right)+\sum_{k=K+1}^{\infty} \frac{1}{2^k}.
\end{align}
Since $\sum_{k=1}^{\infty} \frac{1}{2^k}<\infty$, we derive that
\begin{align*}
	\lim_{\epsilon\to 0} \mathbb{P} \left(\tilde{d}(X^{\epsilon,h_\epsilon},Y^{h_\epsilon})>\rho\right)=0
\end{align*}
for any $\rho>0$ is equivalent to
\begin{align*}
	\lim_{\epsilon\to 0} \mathbb{P}\left(\sup_{t \leq T, x \in \mathbb{R}}\left(\left|X^{\epsilon,h_\epsilon}(t,x)-Y^{h_\epsilon}(t,x)\right|e^{-\lambda|x|e^{\beta t}}\right)>\rho \right) =0
\end{align*}
for any $\lambda,\rho>0$, here and in the following we write $\beta(c_1\vee c_4,\lambda)$ as $\beta$. It is not difficult to see that the above statement
is equivalent to
\begin{align*}
	\lim_{\epsilon\to 0} \mathbb{P}\left(\sup_{t \leq T, x \in \mathbb{R}}\left(\left|X^{\epsilon,h_\epsilon}(t,x)-Y^{h_\epsilon}(t,x)\right|e^{-\lambda|x|e^{\beta t}}\right)>\rho \right) =0
\end{align*}
for any $\rho>0$ and $\lambda\in(0,\lambda_T]$. In the following, we will prove this claim.

Fix $\lambda\in(0,\lambda_T]$. Let $\delta\in(0,\frac{1}{e})$. Introduce stopping times
\begin{align*}
		\tau_{M,\epsilon}:=& \inf \left\{t>0: \sup _{x \in \mathbb{R}}\left(\left|X^{\epsilon,h_\epsilon}(t,x)\right|e^{-\lambda|x| e^{\beta t}}\right) \geq M\right\} \nonumber\\
		& \wedge \inf \left\{t>0: \sup _{x \in \mathbb{R}}\left(\left|Y^{h_\epsilon}(t,x)\right| e^{-\lambda|x| e^{\beta t}}\right) \geq M\right\}, \nonumber\\
		\tau^{\delta,\epsilon}:=& \inf \left\{t>0: \sup _{x \in \mathbb{R}}\left(\left|X^{\epsilon,h_\epsilon}(t,x)-Y^{h_\epsilon}(t,x)\right| e^{-\lambda|x| e^{\beta t}}\right) \geq \delta\right\}, \nonumber \\
		\tau_{M}^{\delta,\epsilon}:=& \tau_{M,\epsilon} \wedge \tau^{\delta,\epsilon}\wedge T.
\end{align*}
It follows from Chebyshev's inequality and Lemma \ref{U} that
\begin{align}\label{221106.1150}
	\lim_{M\rightarrow\infty}\sup_{0<\epsilon<1} \mathbb{P} (\tau_{M,\epsilon}<T) = 0.
\end{align}
Let
\begin{align*}
	Z^{\epsilon}(r):=\sup _{t \leq r \wedge \tau_{M}^{\delta,\epsilon}, x \in \mathbb{R}}\left(\left|X^{\epsilon,h_\epsilon}(t,x)-Y^{h_\epsilon}(t,x)\right| e^{-\lambda|x| e^{\beta t}}\right).
\end{align*}
Then
\begin{align}\label{Z}
	Z^{\epsilon}(r)\leq & \ \sup_{t \leq r \wedge \tau_{M}^{\delta,\epsilon}, x \in \mathbb{R}}\left\{\int_{0}^{t} \int_{\mathbb{R}} p_{t-s}(x, y)\left|b(X^{\epsilon,h_\epsilon}(s,y))-b(Y^{h_\epsilon}(s,y))\right| \mathrm{d} s \mathrm{d} y e^{-\lambda|x| e^{\beta t}}\right\} \nonumber\\
	&+ \sup_{t \leq r \wedge \tau_{M}^{\delta,\epsilon}, x \in \mathbb{R}}\left\{\left|\int_{0}^{t} \int_{\mathbb{R}} p_{t-s}(x, y)\left[\sigma(X^{\epsilon,h_\epsilon}(s,y))-\sigma(Y^{h_\epsilon}(s,y))\right]h_\epsilon(s,y)\mathrm{d} s\mathrm{d} y\right| e^{-\lambda|x| e^{\beta t}}\right\} \nonumber\\
	&+\sqrt{\epsilon} \sup_{t \leq r \wedge \tau_{M}^{\delta,\epsilon}, x \in \mathbb{R}}\left\{\left|\int_{0}^{t} \int_{\mathbb{R}}p_{t-s}(x,y)\sigma(X^{\epsilon,h_\epsilon}(s,y))W(\mathrm{d}s,\mathrm{d}y)\right|e^{-\lambda|x| e^{\beta t}}\right\} \nonumber\\
	=: & \ I + II + III .
\end{align}
It is easy to see that
\begin{align}
	III \leq \sqrt{\epsilon} V^{\epsilon}(T),
\end{align}
here
\begin{align*}
	V^{\epsilon}(T):=\sup_{t \leq T, x \in \mathbb{R}}\left\{\left|\int_{0}^{t} \int_{\mathbb{R}}p_{t-s}(x,y)\sigma(X^{\epsilon,h_\epsilon}(s,y))W(\mathrm{d}s,\mathrm{d}y)\right|e^{-\lambda|x| e^{\beta t}}\right\}.
\end{align*}

By Assumption (\textbf{H1(b)}) and arguments similar to that used to prove (\ref{uniq1})-(\ref{221105.1947}), we have
\begin{align}\label{Z11}
	& I \nonumber\\
\leq & c_{3}\sup _{t \leq r \wedge \tau_{M}^{\delta,\epsilon}, x \in \mathbb{R}}\left\{\int_{0}^{t} \int_{\mathbb{R}} p_{t-s}(x, y) |X^{\epsilon,h_\epsilon}(s,y)-Y^{h_\epsilon}(s,y)|\log_+\frac{1}{|X^{\epsilon,h_\epsilon}(s,y)-Y^{h_\epsilon}(s,y)|}\mathrm{d}s\mathrm{d}y\,e^{-\lambda|x| e^{\beta t}}\right\}\nonumber\\
	&+c_{4}\sup _{t \leq r \wedge \tau_{M}^{\delta,\epsilon}, x \in \mathbb{R}}\!\!\left\{\int_{0}^{t}\! \int_{\mathbb{R}}\! p_{t-s}(x, y)  \log_+ \!\left(|X^{\epsilon,h_\epsilon}(s,y)|\vee|Y^{h_\epsilon}(s,y)|\right)\!|X^{\epsilon,h_\epsilon}(s,y)-Y^{h_\epsilon}(s,y)| \mathrm{d}s\mathrm{d}y\,e^{-\lambda|x| e^{\beta t}}\!\right\} \nonumber\\
	&+ c_{5}\sup _{t \leq r \wedge \tau_{M}^{\delta,\epsilon}, x \in \mathbb{R}}\left\{\int_{0}^{t} \int_{\mathbb{R}} p_{t-s}(x, y) |X^{\epsilon,h_\epsilon}(s,y)-Y^{h_\epsilon}(s,y)|\mathrm{d}s\mathrm{d}y\,e^{-\lambda|x| e^{\beta t}}\right\} \nonumber\\
	\leq& 2c_{3} e^{\frac{\lambda^{2}}{4 \beta} e^{2 \beta T-1}} \int_{0}^{r}  Z^{\epsilon}(s) \log_{+} \frac{1}{Z^{\epsilon}(s)} \mathrm{d} s+2 c_{5} e^{\frac{\lambda^{2}}{4 \beta} e^{2 \beta T-1}} \int_{0}^{r} Z^{\epsilon}(s)\mathrm{d} s \nonumber\\
	&+ c_{4}\sup _{t \leq r \wedge \tau_{M}^{\delta,\epsilon}, x \in \mathbb{R}}\left\{\int_{0}^{t} \sup_{y \in \mathbb{R}}\left(|X^{\epsilon,h_\epsilon}(s,y)-Y^{h_\epsilon}(s,y)| e^{-\lambda|y| e^{\beta s}}\right) \int_{\mathbb{R}} p_{t-s}(x, y) e^{\lambda|y| e^{\beta s}} \lambda|y| e^{\beta s}  \mathrm{d}y\mathrm{d}s \,e^{-\lambda|x| e^{\beta t}}\right\}\nonumber \\
	&+c_{4} \log_{+}M\sup_{t \leq r \wedge \tau_{M}^{\delta,\epsilon}, x \in \mathbb{R}}\left\{\int_{0}^{t} \sup_{y\in\mathbb{R}} \left(|X^{\epsilon,h_\epsilon}(s,y)-Y^{h_\epsilon}(s,y)| e^{-\lambda|y| e^{\beta s}}\right)\int_{\mathbb{R}} p_{t-s}(x, y) e^{\lambda|y| e^{\beta s}} \mathrm{d}y\mathrm{d}s \,e^{-\lambda|x| e^{\beta t}}\right\} \nonumber\\
	\leq &  2c_{3} e^{\frac{\lambda^{2}}{4 \beta} e^{2 \beta T-1}} \int_{0}^{r}  Z^{\epsilon}(s) \log_{+} \frac{1}{Z^{\epsilon}(s)} \mathrm{d} s+2 c_{5} e^{\frac{\lambda^{2}}{4 \beta} e^{2 \beta T-1}} \int_{0}^{r} Z^{\epsilon}(s)\mathrm{d} s\nonumber\\
	&+\frac{c_{4}}{\beta} e^{\frac{\lambda ^{2}}{4 \beta} e^{2 \beta T-1}}Z^{\epsilon}(r)+C_{c_4,\lambda,\beta,M,T}\int_{0}^{r} Z^{\epsilon}(s) \mathrm{d}s \nonumber\\
	\leq & \frac{c_4}{2(c_1\vee c_4)} Z^{\epsilon}(r)+C_{c_1,c_4,c_5,\lambda,M,T}\int_{0}^{r} Z^{\epsilon}(s)\mathrm{d}s+C_{c_1,c_3,c_4,\lambda,T} \int_{0}^{r}  Z^{\epsilon}(s)\log_+\frac{1}{Z^{\epsilon}(s)}\mathrm{d}s.
\end{align}
The last inequality above follows from $\beta=\beta(c_{1}\vee c_{4},\lambda)$, for any $\lambda\in(0,\lambda_T]$, $T\leq T^*(c_1\vee c_4,\lambda)$, and
\begin{align*}
	\frac{c_1\vee c_4}{\beta}e^{\frac{\lambda^2}{4\beta}e^{2\beta T-1}}\leq \frac{1}{2} \iff T\leq T^*(c_1\vee c_4,\lambda).
\end{align*}
By Proposition \ref{moment} and the Lipschitz continuity of $\sigma$ (see (\textbf{H1(c)})), for any $\eta>0$,
\begin{align}\label{Z2}
	II \leq \eta L_\sigma Z^{\epsilon}(r)+C_{\eta,c_1,c_4,\lambda,L_{\sigma},N,T}\int_{0}^{r} \gamma_\epsilon(s)Z^{\epsilon}(s)\mathrm{d}s,
\end{align}
where $\gamma_\epsilon(s)=\int_{\mathbb{R}} |h_\epsilon(s,y)|^2\mathrm{d}y$ and we have used that $\int_0^T \gamma_{\epsilon}(s)\mathrm{d}s \leq N$, $\mathbb{P}$-a.e. due to $\gamma_\epsilon\in\mathcal{H}_N$, $\mathbb{P}$-a.e..
Combining \eqref{Z}-(\ref{Z2}) together leads to, for $r\leq T$,
\begin{align}\label{Zll}
	Z^{\epsilon}(r)\leq & \ \frac{1}{2} Z^{\epsilon}(r)+ C_{c_1,c_4,c_5,\lambda,M,T}\int_{0}^{r} Z^{\epsilon}(s)\mathrm{d}s+C_{c_1,c_3,c_4,\lambda,T} \int_{0}^{r}  Z^{\epsilon}(s)\log_+\frac{1}{Z^{\epsilon}(s)}\mathrm{d}s \nonumber\\
	&+\eta L_\sigma Z^{\epsilon}(r)+C_{\eta,c_1,c_4,\lambda,L_\sigma,N,T}\int_{0}^{r} \gamma_{\epsilon}(s)Z^{\epsilon}(s)\mathrm{d}s +\sqrt{\epsilon}V^{\epsilon}(T).
\end{align}
 Taking $\eta < \frac{1}{2L_{\sigma}}$ and applying Lemma \ref{Gronwall2} gives
\begin{align*}
	Z^{\epsilon}(T)\leq C_{c_1,c_3,c_4,c_5,\lambda,L_\sigma,N,M,T}\left[\sqrt{\epsilon}V^{\epsilon}(T)+(\sqrt{\epsilon}V^{\epsilon}(T))^{C_{c_1,c_3,c_4,\lambda,L_\sigma,T}}\right],
\end{align*}
where $0<C_{c_1,c_3,c_4,\lambda,L_\sigma,T}\leq 1$. Now we take expectations of both sides of the above inequality to obtain
\begin{align}
	\mathbb{E}Z^{\epsilon}(T)\leq  C_{c_1,c_3,c_4,c_5,\lambda,L_\sigma,N,M,T}\left[\sqrt{\epsilon}\mathbb{E}V^{\epsilon}(T)+{\sqrt{\epsilon}}^{C_{c_1,c_3,c_4,\lambda,L_\sigma,T}}\mathbb{E}\left(V^{\epsilon}(T)^{C_{c_1,c_3,c_4,\lambda,L_\sigma,T}}\right)\right].
\end{align}
By the boundedness of $\sigma$, Lemma 4.1 and Proposition 4.2 in \cite{SZ}, we know that for any $p>0$,
\begin{align}\label{2211121.2227}
	\sup_{\epsilon>0}\mathbb{E}[(V^{\epsilon}(T))^p] = \sup_{\epsilon>0}\mathbb{E}\sup_{t \leq T, x \in \mathbb{R}}\left\{\left|\int_{0}^{t} \int_{\mathbb{R}}p_{t-s}(x,y)\sigma(X^{\epsilon,h_\epsilon}(s,y))W(\mathrm{d}s,\mathrm{d}y)\right|e^{-\lambda|x| e^{\beta t}}\right\}^{p} \leq C_{\lambda,\beta,T,p} .
\end{align}
Hence $\mathbb{E} Z^{\epsilon}(T)$ converges to $0$ as $\epsilon\to0$, that is,
\begin{align}
	\lim_{\epsilon\rightarrow 0}\mathbb{E}\sup_{t \leq T \wedge \tau_{M}^{\delta,\epsilon}, x \in \mathbb{R}}\left(\left|X^{\epsilon,h_\epsilon}(t,x)-Y^{h_\epsilon}(t,x)\right| e^{-\lambda|x| e^{\beta t}}\right) = 0.
\end{align}
Applying Chebyshev's inequality yields for any $\rho>0$,
\begin{align}\label{probconv1}
	\lim_{\epsilon\rightarrow 0}\mathbb{P}\Bigg( \sup_{t \leq r \wedge \tau_{M}^{\delta,\epsilon}, x\in \mathbb{R}}\left(\left|X^{\epsilon,h_\epsilon}(t,x)-Y^{h_\epsilon}(t,x)\right| e^{-\lambda|x| e^{\beta t}}\right)>\rho \Bigg) = 0.
\end{align}
Since
\begin{align}
	&\ \mathbb{P}\left(\sup_{t \leq \tau^{\delta,\epsilon}\wedge T, x \in \mathbb{R}}\left(\left|X^{\epsilon,h_\epsilon}(t,x)-Y^{h_\epsilon}(t,x)\right| e^{-\lambda|x| e^{\beta t}}\right)>\rho \right) \nonumber\\
	\leq & \ \mathbb{P}\Bigg( \sup_{t \leq \tau_{M}^{\delta,\epsilon}, x \in\mathbb{R}}\left(\left|X^{\epsilon,h_\epsilon}(t,x)-Y^{h_\epsilon}(t,x)\right| e^{-\lambda|x| e^{\beta t}}\right)>\rho \Bigg) +\mathbb{P}\left(\tau_{M,\epsilon}< T\right) ,
\end{align}
letting $\epsilon\rightarrow 0$ and then $M\rightarrow\infty$ in the above inequality, using \eqref{probconv1} and \eqref{221106.1150}, we get
\begin{align}\label{221107.1203}
	\lim_{\epsilon\rightarrow 0}\mathbb{P}\left(\sup_{t \leq \tau^{\delta,\epsilon}\wedge T, x \in \mathbb{R}}\left(\left|X^{\epsilon,h_\epsilon}(t,x)-Y^{h_\epsilon}(t,x)\right| e^{-\lambda|x| e^{\beta t}}\right)>\rho \right) = 0.
\end{align}
Taking $\rho<\delta$ in the above limit and by the definition of $\tau^{\delta}$, we see that
\begin{align}\label{eq z 01}
	\lim_{\epsilon\to0}\mathbb{P}(\tau^{\delta,\epsilon}<T)=0.
\end{align}
Therefore, by (\ref{221107.1203}) and (\ref{eq z 01}), for any $\rho>0$,
\begin{align}\label{221107.1204}
	&\limsup_{\epsilon\rightarrow 0} \mathbb{P}\left(\sup_{t \leq T, x \in \mathbb{R}}\left(\left|X^{\epsilon,h_\epsilon}(t,x)-Y^{h_\epsilon}(t,x)\right| e^{-\lambda|x| e^{\beta t}}\right)>\rho \right) \nonumber\\
	\leq &\ \limsup_{\epsilon\rightarrow 0} \mathbb{P}\Bigg( \sup_{t \leq \tau^{\delta,\epsilon}\wedge T, x \in\mathbb{R}}\left(\left|X^{\epsilon,h_\epsilon}(t,x)-Y^{h_\epsilon}(t,x)\right| e^{-\lambda|x| e^{\beta t}}\right)>\rho \Bigg) + \limsup_{\epsilon\rightarrow 0}\mathbb{P}\left(\tau^{\delta,\epsilon} < T\right)=0.
\end{align}
The proof of Proposition \ref{C2} is complete.
\end{proof}

\vskip 0.5cm

\section{Verifications of Claims (\textbf{C1}) and (\textbf{C2}): Case 2}\label{Sec 5}

In this section, we assume that

\begin{center}
{\bf Assumption Case 2}:\ \ \ \ \ \ \ \ \  $u_0\in C_{tem}$ and {\rm(\textbf{H0(b)})} holds.
\end{center}
Under this assumption, the aim of this section is to verify Claims (\textbf{C1}) and (\textbf{C2}) in the proof of Theorem \ref{221107.1426}; see Propositions \ref{C1 Case 2} and \ref{C2 Case 2}, respectively.


Since the Lipschitz condition of the coefficient $b$ is equivalent to Assumption {\rm(\textbf{H1(b)})} with $c_3 = c_4 = 0$, most of proofs in Section 4 are still valid for the case considered in this section. What's more, we can take $\beta \equiv 0$ such that $e^{-\lambda|x| e^{\beta t}}$ reads as $e^{-\lambda |x|}$ to simplify the proofs in Section 4. The main difference is the following: $\sigma$ could be unbounded in this section, while it is bounded in Section 4, see {\rm(\textbf{H1(c)})}.
Hence new difficulties occur, naturally.
In the following, we will focus on overcoming the difficulties raised by the unboundedness of $\sigma$.

\vskip 0.2cm

To verify Claim (\textbf{C1}), we first give a priori estimate for (\ref{ske0}); see
the following lemma.
\begin{lemma}\label{221203.2300}
	Let $\{h_n\}_{n\in\mathbb{N}}\subset\mathcal{H}_N$ for some fixed $N>0$.
For $n\geq 1$, let $Y^{h_n}$ be the unique solution of (\ref{ske0}) with $h$ replaced by $h_n$.  	
	Then for any $\lambda,T>0$, we have
	\begin{align}
		\sup_{n \geq 1} \sup_{t \leq T, x \in \mathbb{R}}\left(\left|Y^{h_n}(t, x)\right| e^{-\lambda|x|}\right)\leq C_{\lambda,T,u_0,L,N},
	\end{align}
where $L$ is the Lipschitz constant appearing in  (\ref{221126.2141}).
\end{lemma}
\begin{proof}
	Fix $\lambda>0$.
Using H\"older's inequality, Assumption {\rm(\textbf{H0(b)})} and (\ref{3.21}) we get
	\begin{align}\label{f_n3}
		&\sup_{x\in\mathbb{R}} \left\{ \left|\int_{0}^{t}\int_{\mathbb{R}} p_{t-s}(x,y)\sigma(Y^{h_n}(s,y))h_n(s,y)\mathrm{d}s\mathrm{d}y \right| e^{-\lambda|x|}\right\}\nonumber\\
		\leq & |h_n|_{\mathcal{H}} \sup_{x\in\mathbb{R}}\left\{\left(\int_{0}^{t}\int_{\mathbb{R}}p^2_{t-s}(x,y)\sigma^2(Y^{h_n}(s,y))\mathrm{d}s\mathrm{d}y\right)^\frac{1}{2}e^{-\lambda|x|}\right\}\nonumber\\
		\leq &L|h_n|_{\mathcal{H}} \sup_{x\in\mathbb{R}}\left\{\left(\int_{0}^{t}\int_{\mathbb{R}}p^2_{t-s}(x,y)(1+|Y^{h_n}(s,y)|^2)\mathrm{d}s\mathrm{d}y\right)^\frac{1}{2}e^{-\lambda|x|}\right\}\nonumber\\
		\leq & L|h_n|_{\mathcal{H}} \sup_{x\in\mathbb{R}}\left\{\left(\int_{0}^{t}\frac{1}{\sqrt{\pi(t-s)}}\mathrm{d}s+\int_{0}^{t}\sup_{r\leq s,y\in\mathbb{R}}\left(|Y^{h_n}(r,y)|^2e^{-2\lambda|y|}\right)\int_{\mathbb{R}}p^2_{t-s}(x,y)e^{2\lambda|y|}\mathrm{d}y\mathrm{d}s\right)^\frac{1}{2}e^{-\lambda|x|}\right\}\nonumber\\
		\leq & LN \left(2\sqrt{\frac{t}{\pi}}+ \frac{e^{\lambda^2 t}}{\sqrt{\pi}}\int_{0}^{t}\frac{(f_n(s))^2}{\sqrt{t-s}}\mathrm{d}s\right)^{\frac{1}{2}} .
	\end{align}

Let
	\begin{align}
		f_n(t):=\sup_{x\in\mathbb{R}}\left(|Y^{h_n}(t,x)|e^{-\lambda|x|}\right), \quad t\in[0,T].
	\end{align}

By using the similar argument in the proof of Lemma \ref{u^nbdd} (see also Lemma \ref{u_nbdd}) with $c_3 = c_4=\beta=0$, we can obtain
\begin{align}
	(f_n(t))^2\leq C_{\lambda,t,u_0,L,N}+C_{\lambda,t,L,N}\int_{0}^{t}\frac{(f_n(s))^2}{\sqrt{t-s}}\mathrm{d}s, \quad \ 0\leq t\leq T.
\end{align}
Iterating the above inequality and then applying Gronwall's inequality and taking the supremum over $t\in [0,T]$, we have
\begin{align}
	\sup_{n\geq 1}\sup_{t\leq T,x\in\mathbb{R}}\left(|Y^{h_n}(t,x)|e^{-\lambda|x|}\right)\leq C_{\lambda,T,u_0,L,N}.
\end{align}

The proof of this lemma is complete.
\end{proof}

With the help of Lemma \ref{221203.2300}, using a similar, but simpler, argument in the proof of Proposition \ref{C1}, we have the following proposition, i.e., Claim (\textbf{C1}).
Recall the definition of $\mathcal{G}^0$ introduced in the proof of Theorem \ref{221107.1426}. We have
\begin{proposition}\label{C1 Case 2}
Under {\bf Assumption Case 2}, for any $N>0$, $\{h_n\}_{n\in\mathbb{N}}\subset \mathcal{H}_N$ and $h\in\mathcal{H}_N$ with $h_n\to h$ weakly in $\mathcal{H}$ as $n\to\infty$, then $\lim_{n\rightarrow\infty}\mathcal{G}^0(\text{Int}(h_n))=\mathcal{G}^0(\text{Int}(h))$ in $C([0,T],C_{tem})$.
%
%
\end{proposition}


Let
	\begin{align}\label{eq V}
		V^\epsilon(T)=\sup_{t\leq T,x\in\mathbb{R}}\left|\int_{0}^{t}\int_{\mathbb{R}}p_{t-s}(x,y)\sigma(X^{\epsilon,h_\epsilon}(s,y))W(\mathrm{d}s,\mathrm{d}y)e^{-\lambda|x|}\right| .
	\end{align}

To verify (\textbf{C2}), we will
use the following a priori  estimates.

\begin{lemma}\label{lastbdd}
Assume {\bf Assumption Case 2}. Let $\{h_\epsilon\}_{\epsilon>0}\subset \mathcal{X}_N$ for some $N>0$. Let $X^{\epsilon,h_\epsilon}$ be the unique solution of \eqref{eqnhe}, and $Y^{h_\epsilon}$ be the unique solution of \eqref{ske0} with $h$ replaced by $h_\epsilon$. Then for any $T, \lambda >0$, there exists a constant $C_{\lambda,T,u_0,L,N}$ such that
	\begin{align}
		&\sup_{0<\epsilon \leq 1} \mathbb{E} \bigg[ \sup_{t\leq T,x\in\mathbb{R}}\left(|X^{\epsilon,h_\epsilon}(t,x)|e^{-\lambda|x|}\right)^2 \bigg] \leq C_{\lambda,T,u_0,L,N},\\
	   &\sup_{0<\epsilon \leq 1} \mathbb{E} [ V^{\epsilon}(T)^2 ]\leq C_{\lambda,T,u_0,L,N},	
	\end{align}
where $L$ is the Lipschitz constant appearing in (\ref{221126.2141}).
\end{lemma}
\begin{proof}
	Set
	\begin{align}
		F^\epsilon(t):=\sup_{x\in\mathbb{R}}\left(|X^{\epsilon,h_\epsilon}(t,x)|e^{-\lambda|x|}\right)^2,\ t\in[0,T].
	\end{align}
	Similar to the proof of Lemma \ref{221203.2300}, one deduces that
\begin{align}
		F^\epsilon(t)\leq C_{\lambda,T,u_0,L,N}+C_{\lambda,T,L}\int_{0}^{t}
		F^\epsilon(s)\mathrm{d}s+C_{\lambda,T,L,N}\int_{0}^{t}
		\frac{F^\epsilon(s)}{\sqrt{t-s}}\mathrm{d}s+\epsilon (V^\epsilon(T))^2,\ t\in[0,T].
	\end{align}	
Iterating once, applying Gronwall's inequality and then  taking expectation lead to
	\begin{align}\label{221203.2230}
		\mathbb{E} F^\epsilon(T)\leq C_{\lambda,T,u_0,L,N}\left(\epsilon \mathbb{E} [V^\epsilon(T)^2 ]+ 1\right).
	\end{align}

Now we estimate $\mathbb{E} V^\epsilon(T)$.
According to \cite[Proposition 4.2]{SZ} and by the Lipschitz condition i.e., {\rm(\textbf{H0(b)})}, for any $\eta>0$,
	\begin{align}\label{We}
		\mathbb{E}[ V^\epsilon(T)^2] = &\mathbb{E}\sup_{t\leq T,x\in\mathbb{R}}\left|\int_{0}^{t}\int_{\mathbb{R}}p_{t-s}(x,y)\sigma(X^{\epsilon,h_\epsilon}(s,y))W(\mathrm{d}s,\mathrm{d}y)e^{-\lambda|x|}\right|^2\nonumber\\
		\leq&\eta \mathbb{E}\sup_{t\leq T,x\in\mathbb{R}}\left(|\sigma(X^{\epsilon,h_\epsilon}(t,x))|^2e^{-2\lambda|x|}\right)+C_{\eta,\lambda,T}\mathbb{E}\int_{0}^{T}\int_{\mathbb{R}}|\sigma(X^{\epsilon,h_\epsilon}(t,x))|^2e^{-2\lambda|x|}\mathrm{d}t\mathrm{d}x\nonumber\\
		\leq & \eta L^2 +\eta L^2\mathbb{E} F^{\epsilon}(T)
		+\frac{C_{\eta,\lambda,T}L^2T}{\lambda}+C_{\eta,\lambda,T}L^2\mathbb{E}\int_{0}^{T}\int_{\mathbb{R}}|X^{\epsilon,h_\epsilon}(t,x)|^2e^{-2\lambda|x|}\mathrm{d}t\mathrm{d}x,
	\end{align}
here the constant $C_{\eta,\lambda,T}$ is increasing in $\lambda\in (0,\infty)$.
Set
	\begin{align}
		G^\epsilon(t):=\mathbb{E}\int_{\mathbb{R}}|X^{\epsilon,h_\epsilon}(s,y)|^2e^{-2\lambda|x|}\mathrm{d}x,\ t\in[0,T].
	\end{align}
	By \eqref{eqnhe} we have
	\begin{align}\label{Ge}
		G^\epsilon(t)\leq& 4\int_{\mathbb{R}}\left|\int_{\mathbb{R}}p_{t-s}(x,y)u_0(y)\mathrm{d}y\right|^2e^{-2\lambda|x|}\mathrm{d}x\nonumber\\
		&+4\mathbb{E}\int_{0}^{t}\left|\int_{0}^{t}\int_{\mathbb{R}}p_{t-s}(x,y)b(X^{\epsilon,h_\epsilon}(s,y))\mathrm{d}s\mathrm{d}y\right|^2e^{-2\lambda|x|}\mathrm{d}x\nonumber\\
		&+4\mathbb{E}\int_{0}^{t}\left|\int_{0}^{t}\int_{\mathbb{R}}p_{t-s}(x,y)\sigma(X^{\epsilon,h_\epsilon}(s,y))h_\epsilon(s,y)\mathrm{d}s\mathrm{d}y\right|^2e^{-2\lambda|x|}\mathrm{d}x\nonumber\\
		&+4\epsilon\mathbb{E}\int_{0}^{t}\left|\int_{0}^{t}\int_{\mathbb{R}}p_{t-s}(x,y)\sigma(X^{\epsilon,h_\epsilon}(s,y))W(\mathrm{d}s,\mathrm{d}y)\right|^2e^{-2\lambda|x|}\mathrm{d}x\nonumber\\
		=:&4[G^\epsilon_1(t)+G^\epsilon_2(t)+G^\epsilon_3(t)+\epsilon G^\epsilon_4(t)].
	\end{align}
	By H\"older's inequality, Fubini's theorem, \eqref{3.20}, and $u_{0}\in C_{tem}$,
	\begin{align}
		G^\epsilon_1(t)\leq & \int_{\mathbb{R}} \int_{\mathbb{R}}p_{t-s}(x,y)|u_0(y)|^2\mathrm{d}y \cdot e^{-2\lambda|x|}\mathrm{d}x\nonumber\\
		\leq & 2e^{ 2\lambda^2 t}\int_{\mathbb{R}}|u_0(y)|^2 e^{-2\lambda|y|}\mathrm{d}y, \nonumber\\
		\leq & \frac{4 e^{2\lambda^2 t}}{\lambda} \sup_{y\in\mathbb{R}} \left(|u_0(y)|e^{-\frac{\lambda}{2} |y|}\right)^2<\infty. \label{Ge1}
	\end{align}
Similarly, applying the Lipschitz condition again, we get
	\begin{align}
		G^\epsilon_2(t)\leq &L^2 \mathbb{E}\int_{\mathbb{R}}\int_{0}^{t}\int_{\mathbb{R}}p_{t-s}(x,y)(1+|X^{\epsilon,h_\epsilon}(s,y)|^2)\mathrm{d}s\mathrm{d}y \cdot e^{-2\lambda|x|}\mathrm{d}x \nonumber\\
		\leq & 2L^2 e^{2\lambda^2t}\mathbb{E}\int_{0}^{t}\int_{\mathbb{R}}(1+|X^{\epsilon,h_\epsilon}(s,y)|^2)e^{-2\lambda|y|}\mathrm{d}s\mathrm{d}y\nonumber\\
		\leq & 2L^2 e^{2\lambda^2t}\left(\frac{t}{\lambda}+\int_{0}^{t}G^\epsilon(s)\mathrm{d}s\right),\label{Ge2} \\
		G^\epsilon_3(t)\leq & L^2N^2\mathbb{E}\int_{\mathbb{R}}\int_{0}^{t}\int_{\mathbb{R}}p^2_{t-s}(x,y)(1+|X^{\epsilon,h_\epsilon}(s,y)|^2)\mathrm{d}s\mathrm{d}y \cdot e^{-2\lambda|x|}\mathrm{d}x\nonumber\\
		\leq & \pi^{-\frac{1}{2}}L^2N^2e^{\lambda^2 t}\mathbb{E}\int_{0}^{t}\int_{\mathbb{R}}\frac{1}{\sqrt{t-s}}(1+|X^{\epsilon,h_\epsilon}(s,y)|^2)e^{-2\lambda|y|}\mathrm{d}s\mathrm{d}y\nonumber\\
		\leq & \pi^{-\frac{1}{2}}L^2N^2e^{\lambda^2 t}\left(\frac{2\sqrt{t}}{\lambda}+\int_{0}^{t}\frac{G^\epsilon(s)}{\sqrt{t-s}}\mathrm{d}s\right) , \label{Ge3}
	\end{align}
\begin{align}\label{Ge4}
		G^\epsilon_4(t)\leq & \int_{\mathbb{R}}\mathbb{E}\int_{0}^{t}\int_{\mathbb{R}}p^2_{t-s}(x,y)\sigma^2(X^{\epsilon,h_\epsilon}(s,y))\mathrm{d}s\mathrm{d}y \cdot e^{-2\lambda|x|}\mathrm{d}x\nonumber\\
		\leq& L^2\int_{\mathbb{R}}\mathbb{E}\int_{0}^{t}\int_{\mathbb{R}}p^2_{t-s}(x,y)(1+|X^{\epsilon,h_\epsilon}(s,y)|^2)\mathrm{d}s\mathrm{d}y \cdot e^{-2\lambda|x|}\mathrm{d}x\nonumber\\
		\leq & L^2 \pi^{-\frac{1}{2}}e^{\lambda^2 t}\left(\frac{2\sqrt{t}}{\lambda}+\int_{0}^{t}\frac{G^\epsilon(s)}{\sqrt{t-s}}\mathrm{d}s\right).
	\end{align}
We have used the Fubini theorem and \eqref{3.21} in the last step of (\ref{Ge3}).  And the Fubini theorem and the It\^o isometry have been used to get (\ref{Ge4}).

Combining \eqref{Ge}-\eqref{Ge4} together, for $0\leq t\leq T$ and $0<\epsilon\leq 1$,
	\begin{align}
		G^\epsilon(t)\leq C_{\lambda,t,u_0,L,N}+C_{\lambda,t,L}\int_{0}^{t}G^\epsilon(s)\mathrm{d}s+C_{\lambda,t,L,N}\int_{0}^{t}\frac{G^\epsilon(s)}{\sqrt{t-s}}\mathrm{d}s.
	\end{align}
Iterating the above inequality and then using Gronwall's inequality, we can obtain
	\begin{align}
		G^\epsilon(t)\leq C_{\lambda,T,u_0,L,N},\quad \forall\  0\leq t\leq T.
	\end{align}
	Combining the above inequality with \eqref{We}, we arrive at
	\begin{align}\label{We1}
		\mathbb{E} [V^{\epsilon}(T)^2] \leq  C_{\eta,\lambda,T,u_0,L,N}+\eta L^2\mathbb{E} F^{\epsilon}(T) .
	\end{align}
Substituting \eqref{We1} into (\ref{221203.2230}),
	\begin{align}\label{eq z 02}
		\mathbb{E} F^\epsilon(T) \leq C_{\lambda,T,u_0,L,N}\left[ 1+\epsilon C_{\eta,\lambda,T,u_0,L,N}+\epsilon\eta L^2\mathbb{E} F^\epsilon(T)\right].
	\end{align}
	By choosing $\eta$ small
enough, it follows from (\ref{We1}) and (\ref{eq z 02}) that for any fixed $\lambda>0$,
	\begin{align*}
		\sup_{0<\epsilon\leq 1}\mathbb{E} F^{\epsilon} (T) \leq C_{\lambda,T,u_0,L,N} \quad \text{ and }
		\sup_{0<\epsilon\leq 1}\mathbb{E} [V^{\epsilon}(T)^2] \leq C_{\lambda,T,u_0,L,N}.
	\end{align*}

The proof of Lemma \ref{lastbdd} is complete.
\end{proof}

Now we are in position to state the second main result in this section. The claim (\textbf{C2}) is verified in the following proposition.

\begin{proposition}\label{C2 Case 2}
Under {\bf Assumption Case 2}, for any $N>0$, $\{h_\epsilon,\epsilon>0\}\subset \mathcal{X}_N$ and $\delta>0$,
\begin{align*}
	\lim_{\epsilon\to 0} \mathbb{P} \left(   \tilde{d}\left(X^{\epsilon,h_\epsilon},Y^{h_\epsilon}\right)>\delta  \right)=0,
\end{align*}
where $Y^{h_\epsilon}=\mathcal{G}^{0}(\text{Int}(h_{\epsilon}))$ is the unique solution to (\ref{ske0}) with $h$ replaced by $h_{\epsilon}$, and $X^{\epsilon,h_\epsilon}$ is the unique solution to (\ref{eqnhe}).
\end{proposition}

\begin{proof}
Set
\begin{align}
	Z^{\epsilon}(t):=\sup_{x\in\mathbb{R}}\left(\left|X^{\epsilon,h_\epsilon}(t,x)-Y^{h_\epsilon}(t,x)\right| e^{-\lambda|x|}\right)^2.
\end{align}

Recall $V^\epsilon$ introduced in (\ref{eq V}). An argument similar, but simpler, to that used to prove \eqref{Zll} yields that
\begin{align}\label{221203.2247}
	Z^\epsilon(t)\leq C_{\lambda,t,L}\int_{0}^{t}Z^\epsilon(s)\mathrm{d}s+C_{\lambda,t,L,N}\int_{0}^{t}\frac{Z^\epsilon(s)}{\sqrt{t-s}}\mathrm{d}s+2\epsilon (V^\epsilon(t))^2, \quad t\in [0,T].
\end{align}
Iterating (\ref{221203.2247}) and then applying Gronwall's inequality and taking expectations, we have
\begin{align}
	\mathbb{E}Z^\epsilon(T)\leq \epsilon C_{\lambda,T,L,N}\mathbb{E} [V^\epsilon(T)^2].
\end{align}
By Lemma \ref{lastbdd} we know that $\mathbb{E}Z^\epsilon(T)\to 0$ as $\epsilon\to 0$.

The proof of this proposition is complete.


\end{proof}

\section{Appendix}
\setcounter{equation}{0}

In the appendix, we present two Gronwall-type inequalities (see Lemmas \ref{Gronwall1} and \ref{Gronwall2}) and some a priori estimates with respect to the heat kernel (see Lemma \ref{lem 6.3}). 

\begin{lemma}\cite[Theorem 3.1]{W}\label{Gronwall1}
	Let $x(t), a(t), c_{1}(t), c_{2}(t)$ be nonnegative functions on $[0,T]$ and $c_0\ge 1$ be a constant. If for any $t\in[0,T]$,
	\begin{align*}
		x(t)+a(t)\leq c_0+\int_{0}^{t} c_{1}(s)x(s) \mathrm{d} s+\int_{0}^{t} c_{2}(s) x(s) \log_{+} x(s) \mathrm{d}s,
	\end{align*}
	and the integrals above are finite. Then for any $t\in[0,T]$,
	\begin{align*}
		x(t)\leq c_0^{e^{\int_{0}^{t} c_{2}(s) \mathrm{d} s}}
		\exp \left\{ e^{\int_{0}^{t} c_{2}(s) \mathrm{d}s} \int_{0}^{t} c_{1}(s) e^{-\int_{0}^{s} c_{2}(r)\mathrm{d}r} \mathrm{d} s \right\}.
	\end{align*}
\end{lemma}

\begin{lemma}\label{Gronwall2}
	Let $x(t),c_1(t),c_2(t)$ be nonnegative functions on $[0,T]$ and $c_0$ be a nonnegative constant.
	If for any $t\in[0,T]$,
	\begin{align*}
		x(t)\le c_0+\int_{0}^{t} c_1(s)x(s)ds+\int_{0}^{t} c_2(s)x(s)\log_+\frac{1}{x(s)}\mathrm{d}s,
	\end{align*}
	and the integrals above are finite, then for any $t\in[0,T]$,
	\begin{align*}
		x(t)\le \left(c_0+c_0^{e^{-\int_{0}^{t} c_2(s)\mathrm{d}s} } \right) e^{\int_{0}^{t}\left(c_1(s)+c_2(s)\right)\mathrm{d}s}.
	\end{align*}
\end{lemma}

\begin{proof}
	Let
	\begin{align*}
		y(t):= c_0+\int_{0}^{t} c_1(s)x(s)ds+\int_{0}^{t} c_2(s)x(s)\log_+\frac{1}{x(s)}\mathrm{d}s.
	\end{align*}

The monotonicity of the function $z\mapsto z\log_+ \frac{1}{z}$ on $[0,\frac{1}{e})$ implies that
\begin{eqnarray*}
\text{if }0\le y(t)< \frac{1}{e},\text{ then }x(t)\log_+\frac{1}{x(t)}\leq y(t)\log_+\frac{1}{y(t)}.
\end{eqnarray*}
Combining the fact that if $y(t)\ge \frac{1}{e}$, then $\max_{z\ge 0} z\log_+\frac{1}{z} =\frac{1}{e}\leq y(t)$, we have
	\begin{align}\label{y de}
		y'(t) = &\  c_1(t)x(t)+c_2(t)x(t)\log_+\frac{1}{x(t)}\nonumber\\
		\leq &\  c_1(t)y(t)+c_2(t) \left(y(t)+y(t)\log_+\frac{1}{y(t)}\right)\nonumber\\
		= &\  (c_1(t)+c_2(t)) y(t)+c_2(t)y(t) \log_+\frac{1}{y(t)}.
	\end{align}
	Set $c_3(t):=c_1(t)+c_2(t)$ and $t_0:=\inf\{t\geq0:y(t)=1\}$. The above inequality implies that, for $t\in [0,t_0]$,
	\begin{align*}
		(\log y(t))' +c_2(t)\log y(t)\leq c_3(t).
	\end{align*}
	Hence, for $t\in [0,t_0]$,
	\begin{align*}
		\log y(t)\le \left(\log c_0+\int_{0}^{t} c_3(s)e^{\int_{0}^{s} c_2(r)\mathrm{d}r}\mathrm{d}s\right)e^{-\int_{0}^{t} c_2(s)\mathrm{d}s}.
	\end{align*}
	It follows that, for any $t\in [0,t_0]$,
	\begin{align}\label{gron1}
		y(t)\leq &	c_0^{e^{-\int_{0}^{t} c_2(s)\mathrm{d}s}}\exp\left\{\int_{0}^{t} c_3(s)e^{\int_{0}^{s} c_2(r)\mathrm{d}r}\mathrm{d}s\,e^{-\int_{0}^{t} c_2(s)\mathrm{d}s}\right\} \nonumber\\
		\leq & c_0^{e^{-\int_{0}^{t} c_2(s)\mathrm{d}s}}e^{\int_{0}^{t} c_3(s)\mathrm{d}s} .
	\end{align}
	Since $y(t), t\geq0$ is increasing due to $y'(t)\ge 0$, the definition of $t_0$ and (\ref{y de}) imply that, for $t\in[t_0,T]$,
	\begin{align*}
		y'(t)\leq c_3(t)y(t).
	\end{align*}
    Hence \begin{align}\label{gron2}
		y(t)\leq y(t_0)e^{\int_{t_0}^{t} c_3(s)\mathrm{d}s}.
	\end{align}
	Combining (\ref{gron2}) with (\ref{gron1}) leads to
	\begin{align*}
		y(t)\leq& \left(c_0+ c_0^{e^{-\int_{0}^{t} c_2(s)\mathrm{d}s}} \right) e^{\int_{0}^{t} c_3(s)\mathrm{d}s},\ \forall t\in[t_0,T].
	\end{align*}
Here we remark that, to get the above inequality, we distinguish between the cases $c_0\geq 1$ and $0\leq c_0<1$.
Thus, for any $t\in[0,T]$,
	\begin{align*}
		x(t)\leq y(t)\leq \left(c_0+ c_0^{e^{-\int_{0}^{t} c_2(s)\mathrm{d}s}} \right) e^{\int_{0}^{t} (c_1(s)+c_2(s))\mathrm{d}s}.
	\end{align*}

The proof of this lemma is complete.
\end{proof}

Some a priori estimates of the heat kernel are given as follows.
\begin{lemma}\label{lem 6.3}
	(\romannumeral1) For any $x \in \mathbb{R}$, $t>0$ and $\eta \in \mathbb{R}$,
	\begin{align}\label{3.20}
		\int_{\mathbb{R}} p_{t}(x, y) e^{\eta|y|} \mathrm{d} y \leq 2 e^{\frac{\eta^{2} t}{2}} e^{\eta|x|}.
	\end{align}
	
	(\romannumeral2) For any $x\in \mathbb{R}$, $t>0$ and $\eta \in \mathbb{R}$,
	\begin{align}\label{3.21}
		\int_{\mathbb{R}} p_{t}^2(x, y) e^{\eta|y|} \mathrm{d} y \leq \frac{1}{\sqrt{\pi t}} e^{\frac{\eta^{2} t}{4}} e^{\eta|x|}.
	\end{align}
	
	(\romannumeral3) For any $x\in \mathbb{R}$ and $t,\eta >0$,
	\begin{align}\label{3.22}
		\int_{\mathbb{R}} p_{t}(x, y) e^{\eta|y|} \eta|y| \mathrm{d} y \leq e^{\frac{\eta^{2} t}{2}} e^{\eta|x|} \eta|x|+2 e^{\frac{\eta^{2} t}{2}}\left(\eta^{2} t+\eta \sqrt{\frac{t}{2 \pi}}\right) e^{\eta|x|}.
	\end{align}
	
	(\romannumeral4) For any $x, y \in \mathbb{R}, \theta \in[0,1], 0<s \leq t$,
	\begin{align}\label{1oflem3.3}
		\left|p_{t}(x, y)-p_{s}(x, y)\right| \leq \frac{(2 \sqrt{2})^{\theta}|t-s|^{\theta}}{s^{\theta}}\left(p_{s}(x, y)+p_{t}(x, y)+p_{2 t}(x, y)\right).
	\end{align}
	
	(\romannumeral5) For any $x, y \in \mathbb{R}$ and $t>0$,
	\begin{align}\label{2oflem3.3}
		\int_{\mathbb{R}}\left|p_{t}(x, z)-p_{t}(y, z)\right| \mathrm{d} z \leq \sqrt{\frac{2}{\pi}} \frac{|x-y|}{\sqrt{t}} .
	\end{align}
	
	(\romannumeral6) For any $x, y \in \mathbb{R}$ and $\eta, t>0$,
	\begin{align}\label{3oflem3.3}
		\int_{\mathbb{R}}\left|p_{t}(x, z)-p_{t}(y, z)\right| e^{\eta|z|} \mathrm{d} z \leq 2 \sqrt{2} \frac{|x-y|}{\sqrt{t}}  e^{\eta^{2} t} e^{\eta(|x|+|x-y|)} .
	\end{align}

	(\romannumeral7) For any $x, y \in \mathbb{R}$ and $\eta, t>0$,
	\begin{align}\label{4oflem3.3}
		& \int_{\mathbb{R}}\left|p_{t}(x, z)-p_{t}(y, z)\right| e^{\eta|z|} \eta|z| \mathrm{d} z \nonumber \\
		\leq & \frac{\sqrt{2}|x-y|}{\sqrt{t}} \left[e^{\eta^{2} t}  e^{\eta(|x|+|x-y|)} \eta(|x|+|x-y|)+2 e^{\eta^{2} t}\left(2 \eta^{2} t+\eta \sqrt{\frac{t}{\pi}}\right) e^{\eta(|x|+|x-y|)}\right].
	\end{align}
	
	(\romannumeral8) For any $x, y \in \mathbb{R}$ and $0<s \leq t$,
	\begin{align}\label{5oflem3.3}
		\int_{0}^{s} \int_{\mathbb{R}}\left|p_{t-r}(x, z)-p_{s-r}(y, z)\right|^{2} \mathrm{d} r \mathrm{d} z \leq \frac{\sqrt{2}-1}{\sqrt{\pi}}|t-s|^{\frac{1}{2}}+\frac{2}{\sqrt{\pi}}|x-y|.
	\end{align}
\end{lemma}
It is easy to verify (\romannumeral1)-(\romannumeral3). And (\romannumeral4)-(\romannumeral8) can be found in  Lemma 3.3 of \cite{SZ}.

\end{document}